\documentclass[a4paper,12pt]{article}

\usepackage{amsfonts,amssymb,amsmath,amsthm,latexsym,epsfig,amscd,bbm,stmaryrd}

\usepackage[english]{babel}
\usepackage{graphicx}
\usepackage{tikz}
\usepackage{courier}
\usepackage{bbm}
\usepackage{enumerate}
\usepackage{psfrag}
\usepackage{wasysym}
\usepackage{type1cm}

\def\mtiny#1{\scriptscriptstyle{#1}}

\usepackage{calc}

\def\arraypar#1{\parbox[c]{\textwidth - 2cm}{\centering #1}}

\usepackage{environ}
\NewEnviron{display}{
\begin{equation}\begin{array}{c}
\arraypar{\BODY}
\end{array}\end{equation}
}

\def\clap#1{\hbox to 0pt{\hss#1\hss}}

\def\mathclap{\mathpalette\mathclapinternal}

\def\mathclapinternal#1#2{\clap{$\mathsurround=0pt#1{#2}$}}

\makeatletter \@addtoreset{equation}{section}
\makeatother

\makeatletter \@addtoreset{enunciato}{section}
\makeatother

\newcounter{enunciato}[section]

\newtheorem{ittheorem}{Theorem}
\newtheorem{itlemma}{Lemma}
\newtheorem{itproposition}{Proposition}
\newtheorem{itdefinition}{Definition}
\newtheorem{itremark}{Remark}
\newtheorem{itclaim}{Claim}
\newtheorem{itfact}{Fact}
\newtheorem{itconjecture}{Conjecture}
\newtheorem{itcorollary}{Corollary}

\newenvironment{theorem}{\addtocounter{enunciato}{1}
\begin{ittheorem}}{\end{ittheorem}}

\newenvironment{lemma}{\addtocounter{enunciato}{1}
\begin{itlemma}}{\end{itlemma}}

\newenvironment{proposition}{\addtocounter{enunciato}{1}
\begin{itproposition}}{\end{itproposition}}

\newenvironment{definition}{\addtocounter{enunciato}{1}
\begin{itdefinition}}{\end{itdefinition}}

\newenvironment{remark}{\addtocounter{enunciato}{1}
\begin{itremark}}{\end{itremark}}

\newenvironment{conjecture}{\addtocounter{enunciato}{1}
\begin{itconjecture}}{\end{itconjecture}}

\newenvironment{corollary}{\addtocounter{enunciato}{1}
\begin{itcorollary}}{\end{itcorollary}}

\newcommand{\be}[1]{\begin{equation}\label{#1}}
\newcommand{\ee}{\end{equation}}

\newcommand{\bl}[1]{\begin{lemma}\label{#1}}
\newcommand{\el}{\end{lemma}}

\newcommand{\br}[1]{\begin{remark}\label{#1}}
\newcommand{\er}{\end{remark}}

\newcommand{\bt}[1]{\begin{theorem}\label{#1}}
\newcommand{\et}{\end{theorem}}

\newcommand{\bd}[1]{\begin{definition}\label{#1}}
\newcommand{\ed}{\end{definition}}

\newcommand{\bp}[1]{\begin{proposition}\label{#1}}
\newcommand{\ep}{\end{proposition}}

\newcommand{\bc}[1]{\begin{corollary}\label{#1}}
\newcommand{\ec}{\end{corollary}}

\newcommand{\bcj}[1]{\begin{conjecture}\label{#1}}
\newcommand{\ecj}{\end{conjecture}}

\newcommand{\bpr}{\begin{proof}}
\newcommand{\epr}{\end{proof}}

\DeclareMathOperator\supp{supp}
\DeclareMathOperator\dist{dist}

\DeclareMathOperator\Cov{Cov}
\DeclareMathOperator\Exp{Exp}
\DeclareMathOperator\Trace{Trace}
\DeclareMathOperator\Poisson{Poisson}
\DeclareMathOperator\Uniform{Uniform}

\def\d{\mathrm{d}}
\def\m{m}

\def\Z{\mathbb{Z}}
\def\N{\mathbb{N}}
\def\R{\mathbb{R}}

\def\P{\mathbb{P}}
\def\E{\mathbb{E}}

\newcommand{\1}[1]{{\mathbbm{1}}_{#1}}

\def \ba {\begin{array}}
\def \ea {\end{array}}

\def \P  {{\mathbb P}}
\def \E  {{\mathbb E}}

\def \cW {{\mathcal W}}
\def \cT {{\mathcal T}}

\usepackage{graphicx}

\DeclareSymbolFont{symbolsC}{U}{pxsyc}{m}{n}
\DeclareMathSymbol{\opentimes}{\mathrel}{symbolsC}{93}

\newcommand{\um}{{\angle}}
\newcommand{\tres}{{\mathbin{\, \text{\rotatebox[origin=c]{180}{$\angle$}}}}}
\newcommand{\treze}{{\mathbin{\text{\rotatebox[origin=c]{35}{$\opentimes$}}}}}

\newcounter{constant}
\newcommand{\newconstant}[1]{\refstepcounter{constant}\label{#1}}
\newcommand{\useconstant}[1]{c_{\textnormal{\tiny \ref{#1}}}}
\usepackage{array}
\newcolumntype{e}{>{\displaystyle}r @{\,} >{\displaystyle}c @{\,} >{\displaystyle}l}

\begin{document}


\title{Random Walk on Random Walks}

\author{\renewcommand{\thefootnote}{\arabic{footnote}}
M.\ Hil\'ario\footnotemark[1] \footnotemark[2]  , \;
\renewcommand{\thefootnote}{\arabic{footnote}}
F.\ den Hollander\footnotemark[3] , \;
\renewcommand{\thefootnote}{\arabic{footnote}}
R.S.\ dos Santos\footnotemark[4] , \;\\
\renewcommand{\thefootnote}{\arabic{footnote}}
V.\ Sidoravicius\footnotemark[5] \footnotemark[6] \footnotemark[7]
\renewcommand{\thefootnote}{\arabic{footnote}}
A.\ Teixeira\footnotemark[5]}

\footnotetext[1]{Universidade Federal de Minas Gerais, Dep. de
Matem\'atica, 31270-901 Belo Horizonte}
\footnotetext[2]{Universit\'e de Gen\`eve, Section de Math\'ematiques, 2-4 Rue du Li\`evre, 1211 Gen\`eve}

\footnotetext[3]{Mathematical Institute, Leiden University, P.O.\ Box 9512,
2300 RA Leiden}

\footnotetext[4]{Weierstrass Institute for Applied Analysis and Stochastics, Mohrenstr.\ 39, 10117 Berlin}

\footnotetext[5]{Instituto Nacional de Matem\'atica Pura e Aplicada,
Estrada Dona Castorina 110, 22460-320 Rio de Janeiro}

\footnotetext[6]{Courant Institute, NYU, 251 Mercer Street New York, NY 10012}
\footnotetext[7]{NYU-Shanghai, 1555 Century Av., Pudong Shanghai, CN 200122}

\maketitle

\begin{abstract}
In this paper we study a random walk in a one-dimensional dynamic random environment
consisting of a collection of independent particles performing simple symmetric random walks 
in a Poisson equilibrium with density $\rho \in (0,\infty)$. At each step the random walk  
performs a nearest-neighbour jump, moving to the right with probability $p_{\circ}$ when it is 
on a vacant site and probability $p_{\bullet}$ when it is on an occupied site. Assuming 
that $p_\circ \in (0,1)$ and $p_\bullet \neq \tfrac12$, we  show that the position of the random 
walk satisfies a strong law of large numbers, a functional central limit theorem and a large 
deviation bound, provided $\rho$ is large enough. The proof is based on the construction of 
a renewal structure together with a multiscale renormalisation argument.

\vspace{0.5cm}
\noindent {\it MSC} 2010.
Primary 60F15, 60K35, 60K37; Secondary 82B41, 82C22, 82C44.
\end{abstract}


\section{Introduction and main results}
\label{s:intro}

{\bf Background.}
Random motion in a random medium is a topic of major interest in mathematics, physics 
and (bio-)chemistry. It has been studied at microscopic, mesoscopic and macroscopic 
levels through a range of different methods and techniques coming from numerical, 
theoretical and rigorous analysis.

Since the pioneering work of Harris \cite{Ha}, there has been much interest in studies of 
random walk in random environment within probability theory (see \cite{HMZ} for an overview), 
both for static and dynamic random environments, and a number of deep results have been
proven for various types of models.

In the case of dynamic random environments, analytic, probabilistic and ergodic 
techniques were invoked (see e.g.\ \cite{AvdHoRe11}, \cite{BhZtn}, \cite{BoCoFrPe05}--\cite{BoMiPe09},  \cite{dkl}, \cite{dk}, \cite{JoRa11}, \cite{rghspp}, \cite{ReVo13}), 
but \emph{good mixing assumptions on the environment} remained  a pivotal requirement. 
By good mixing we mean that the decay of space-time correlations is sufficiently fast -- 
polynomial with a sufficiently large degree -- and uniform in the initial configuration. More 
recently, examples of dynamic random environments with non-uniform mixing have been 
considered (see e.g.\ \cite{dHodSa}, \cite{MoVa}, \cite{BiCeDeGa13}, \cite{AvdSaVo}). However, in all of these 
examples either the mixing is fast enough (despite being non-uniform), or the mixing is 
slow but strong extra conditions on the random walk  are required.

In this context, random environments consisting of a field of random walks moving independently 
gained significance, not only due to an abundance of models defined in this setup,  but also due 
to the substantial mathematical challenges that arise from their study. Among various conceptual 
and technical difficulties, slow mixing (in other words, slow convergence of the environment to its 
equilibrium as seen from the walk) makes the analysis of these systems extremely 
difficult. In particular, in physical terms, when ballistic behaviour occurs the motion of the walk 
is of ``pulled type'' (see \cite{Sar}).

In this paper we consider a dynamic random environment given by a system of independent 
random walks. More precisely, we consider a \emph{walk particle} that performs a discrete-time 
motion on $\Z$ under the influence of a field of \emph{environment particles} which themselves 
perform independent discrete-time simple random walks. As initial state for the environment 
particles we take an i.i.d.\ Poisson random field with mean $\rho \in (0,\infty)$. This makes the 
dynamic random environment invariant under translations in space and time. The jumps of the 
walk particle are drawn from two different random walk transition kernels on $\Z$, depending 
on whether the space-time position of the walk particle is occupied by an environment particle 
or not. For reasons of exposition we restrict to nearest-neighbour kernels, but our analysis 
easily extends to the case where the kernels have finite range.

\medskip\noindent
{\bf Model.}
Throughout the paper we write $\N=\{1,2,\dots\}$, ${\mathbb{Z}_+} = \N \cup \{0\}$ and ${\mathbb{Z}_-} 
= -\N \cup \{0\}$. Let $\{N(x)\colon\,x\in\Z\}$ be an i.i.d.\ sequence of Poisson random variables with 
mean $\rho \in (0,\infty)$. At time $n=0$, for each $x\in\Z$ place $N(x)$ environment particles at site 
$x$. Subsequently, let all the environment particles evolve independently as ``lazy simple random walks'' 
on $\Z$, i.e., at each unit of time the probability to step $-1$, $0$, $1$ equals $\tfrac12(1-q)$, $q$, 
$\tfrac12(1-q)$, respectively, for some $q \in (0,1)$. The assumption of laziness  is not crucial for our 
arguments, as explained in Comment 5 below.

Let $\mathcal{T}$ be the set of space-time points covered by the trajectory of at least one environment 
particle. The law of $\mathcal{T}$ is denoted by $P^\rho$ (see Section~\ref{ss:DRE} for a detailed 
construction of the dynamic environment and the precise definition of $\mathcal{T}$).  Note 
that $\mathcal{T}$  does not have good mixing properties. Indeed,
\begin{equation}\label{e:cov}
\Cov_\rho (\mathbbm{1}_{(0,0) \in \mathcal{T}}, \mathbbm{1}_{(0,n) \in \mathcal{T}}) \sim c(\rho) \frac{1}{n^{1/2}},
\end{equation}
where $\Cov_\rho$ denotes covariance with respect to $P^{\rho}$ (see 
\eqref{e:justcov1}--\eqref{e:justcov3} in Section~\ref{ss:DRE}.)

Given $\mathcal{T}$, let $X=(X_n)_{n\in{\mathbb{Z}_+}}$ be the nearest-neighbour random walk on 
$\Z$ starting at the origin and with transition probabilities
\begin{equation}
\label{e:trpr}
P^\mathcal{T}(X_{n+1}=x+1 \mid X_n = x) = \left\{\begin{array}{ll}
p_{\circ}, &\text{ if }(x,n) \notin \mathcal{T},\\
p_{\bullet}, &\text{ if } (x,n) \in \mathcal{T},
\end{array}
\right.
\end{equation}
where $p_{\circ},p_{\bullet} \in [0,1]$ are fixed parameters and $P^\mathcal{T}$ stands for the law of 
$X$ conditional on $\mathcal{T}$, called the \emph{quenched law}. The \emph{annealed law} is given 
by $\P^\rho(\cdot) = \int P^\mathcal{T}(\cdot)\,P^\rho(d\mathcal{T})$.

We will denote by
\begin{equation}
\label{e:defdrifts}
v_{\circ}=2p_{\circ}-1 \quad \text{ and } \quad v_{\bullet}=2p_{\bullet}-1
\end{equation}
the drifts at vacant and occupied sites, respectively. 
Following the terminology established in the literature on random walks in static random environments (see \cite{Sz2},  \cite{Ze98}),
we classify our model as follows.
\begin{definition}
The model is said to be non-nestling when $v_\circ v_\bullet > 0$. Otherwise it is said 
to be nestling.
\end{definition}

We are  now in the position to state our main results.

\begin{theorem}
\label{thm:main}
Let $v_\bullet \neq 0$ and  $v_\circ \neq -\mathrm{sign}(v_\bullet)$. Then there exist 
$\rho_\star \geq 0$ and $\gamma > 1$ such that for all $\rho \geq \rho_\star$ there exist 
$v = v(v_\circ, v_\bullet, \rho) \in [v_\circ \wedge v_\bullet, v_\circ \vee v_\bullet]$ and 
$\sigma = \sigma(v_\circ, v_\bullet, \rho) \in (0, \infty)$ such that:

\medskip\noindent 
(a) $\mathbb{P}^\rho$-almost surely,
\begin{equation}
\label{e:SLLN}
\lim_{n\to\infty} n^{-1} X_n = v.
\end{equation}

\noindent
(b) Under $\P^\rho$, the sequence of random processes
\begin{equation}
\label{e:CLT}
\left(\frac{X_{\lfloor n t \rfloor} - \lfloor nt \rfloor v}{n^{1/2}\sigma}\right)_{t \ge 0},
\qquad n\in\N,
\end{equation}
converges in distribution (in the Skorohod topology) to the standard Brownian motion.

\medskip\noindent
(c) For all $\varepsilon > 0$ there exists $c = c(v_\circ, v_\bullet, \rho, \varepsilon) \in (0,\infty)$ 
such that
\begin{equation}
\label{e:LD2}
\P^\rho \big( \exists \; t \geq n \colon\, |X_t - tv| > \varepsilon t \big)
\leq c^{-1} e^{-c \log^\gamma n} \qquad \forall \; n \in \N.
\end{equation}

\noindent
Moreover, in the non-nestling case $\rho_\star$ can be taken equal to $0$.
\end{theorem}

\noindent
The difference between the nestling and the non-nestling case can be seen in the statement 
of Theorem~\ref{thm:main}: in the non-nestling case we can prove (a)--(c) for any 
$\rho \ge 0$, in the nestling case only for $\rho \geq \rho_\star$, where $\rho_\star$ will need
to be large enough.

Theorem~\ref{thm:main} will be obtained as a consequence of 
Theorems~\ref{thm:SLLN+CLT+LD}--\ref{thm:LD} and Remark~\ref{r:v_leq_v} below. 
Before stating them, we give the following definition that will be central to our analysis.
\begin{definition}
\label{def:BC}
For fixed $v_\circ, v_\bullet, \rho$ and a given $v_\star \in [-1,1]$, we say that the $v_\star$-ballisticity 
condition holds when there exist $c = c(v_\circ, v_\bullet, v_\star,\rho)> 0$ and $\gamma = \gamma
(v_\circ, v_\bullet, v_\star, \rho) >  1$ such that
\begin{equation}
\label{e:LD}
\P^\rho\big(\exists\, n \in \N\colon\, X_n < n v_\star-L\big) \leq c^{-1} e^{ -c \log^\gamma L }
\qquad \forall\,L\in\N.
\end{equation}
\end{definition}

\noindent
Condition \eqref{e:LD} is reminiscent of ballisticity conditions in the literature on random walks in static 
random environments, such as Sznitman's \emph{$(T')$-condition} (see \cite{Sz2}).

The next theorem shows that, if the model satisfies \eqref{e:LD} with $v_\star > 0$ as well as an 
ellipticity condition, then the asymptotic results stated in Theorem~\ref{thm:main} hold.

\begin{theorem}
\label{thm:SLLN+CLT+LD}
Let $v_\bullet > 0$, $v_\circ>-1$ and $\rho \in (0,\infty)$. Assume that \eqref{e:LD} holds for some 
$v_\star \in (0,1]$. Then the conclusions of Theorem~{\rm \ref{thm:main}} hold with $v \ge v_\star$.
\end{theorem}

Our last theorem shows that \eqref{e:LD} holds when $v_\circ \le v_\star < v_\bullet$ and 
$\rho$ is large enough.

\begin{theorem}
\label{thm:LD}
If $v_\circ < v_\bullet$, then for all $v_\star \in [v_\circ,v_\bullet)$ there exist $\rho_\star
= \rho_\star(v_\circ, v_\bullet, v_\star) \in (0,\infty)$ and $c=c(v_\circ, v_\bullet,v_\star) \in
(0,\infty)$ such that \eqref{e:LD} holds with $\gamma = \tfrac32$ for all $\rho \geq \rho_\star$.
\end{theorem}

\begin{remark}
\label{r:v_leq_v}
If $v_\circ \wedge v_\bullet > 0$, then \eqref{e:LD} holds for all $\rho \in (0,\infty)$ and $v_\star
\in (0,v_\circ \wedge v_\bullet)$ by comparison with a homogeneous random walk with drift
$v_\circ \wedge v_\bullet$. In fact, in this case the bound in the right-hand side of \eqref{e:LD}
can be made exponentially small in $L$.
\end{remark}

Theorem~\ref{thm:main} now follows directly from Theorems~\ref{thm:SLLN+CLT+LD}--\ref{thm:LD} 
and Remark~\ref{r:v_leq_v} by noting that, when $v_\bullet \neq 0$, by reflection symmetry we may 
without loss of generality assume that $v_\bullet >0$. Note that Theorem~\ref{thm:LD} is only needed 
in the nestling case.

The proofs of Theorems~\ref{thm:SLLN+CLT+LD} and \ref{thm:LD} are given in 
Sections~\ref{s:regeneration} and \ref{s:renorm}, respectively. They rely on the construction 
and control of a \emph{renewal structure} for the random walk trajectory, respectively, on a 
\emph{multiscale renormalisation scheme}. The latter is used to show that the random walk 
stays to the right of a point that moves at a strictly positive speed. The former is used to show 
that, as a consequence of this ballistic behaviour, the random walk has a tendency to
outrun the environment particles, which only move diffusively, and to enter ``fresh territory'' 
containing particles it has never encountered before. Therefore the random walk trajectory 
is a concatenation of ``large independent random pieces'', and this forms the basis on which 
the limit laws in Theorem~\ref{thm:main} can be deduced (after appropriate tail estimates). 
None of these techniques are new in the field, but in the context of slow mixing dynamic 
random environments they are novel and open up gates to future advances.

\medskip\noindent 
{\bf Comments.}

\noindent 
{\bf 1.} It follows from Theorem~\ref{thm:LD} (and reflection symmetry in the case $v_\bullet < v_\circ$) 
that
\begin{equation}
\label{e:limit_v}
\lim_{\rho \to \infty} v(v_\circ, v_\bullet, \rho) = v_\bullet,
\end{equation}
where $v = v(v_\circ, v_\bullet, \rho)$ is as in \eqref{e:SLLN}.  This can also be deduced 
from the following asymptotic weak law of large numbers derived in \cite{dHoKeSi}:
\begin{equation}
\lim_{\rho\to\infty} \limsup_{n\to\infty} \P^\rho\big(|n^{-1}X_n-v_{\bullet}|>\varepsilon\big)
= 0 \qquad \forall\,\varepsilon>0.
\end{equation}
In fact, \cite{dHoKeSi} considers the version of our model in $\Z^d, d\ge1$, in continuous
time and with more general transition kernels. 

\medskip\noindent
{\bf 2.}
It can be shown that the asymptotic speed and variance $v$ and $\sigma$ in Theorem~\ref{thm:main}
are continuous functions of the parameters $v_\circ$, $v_\bullet$ and $\rho$. (See Remark~\ref{r:continuity} 
in Section~\ref{ss:limittheorems}.) 

\medskip\noindent
{\bf 3.} 
We expect Theorem~\ref{thm:main} to hold when $v_\circ \neq 0$, $v_\bullet \neq - \mathrm{sign}
(v_\circ)$ and $\rho$ is small. In the non-nestling case, this already follows (for any $\rho \ge 0$) from 
Theorem~\ref{thm:SLLN+CLT+LD}, but in the nestling case we would have to prove the analogue of 
Theorem~\ref{thm:LD} for $v_\bullet < v_\circ$ and $\rho$ small.

\medskip\noindent
{\bf 4.}
Our techniques can potentially be extended to higher dimensions. The restriction to the one-dimensional 
setting simplifies the notation and allows us to avoid certain technicalities.

\medskip\noindent
{\bf 5.}
Our dynamic random environment is composed of lazy random walks evolving in discrete time. This 
assumption was made for convenience in order to simplify some technical steps. However, as discussed 
in Remark~\ref{r:lazy2}, our analysis can be extended to symmetric random walks with bounded steps 
that are aperiodic or bipartite (in the sense of  \cite{LaLi}), or that evolve in continuous time.

\medskip\noindent
{\bf 6.}
It is a challenge to extend Theorem \ref{thm:main} to other environments, in particular to environments 
where the particles are allowed to interact with each other. The renormalisation scheme is robust enough 
to show that the ballisticity condition \eqref{e:LD} holds as long as the environment satisfies a mild 
decoupling inequality (see Section~\ref{ss:commentsreg} for specific examples). On the other hand, 
the regeneration structure is more delicate, and uses model-specific features in an important way. 
A recent development can be found in \cite{HS}, where the authors consider as dynamic environment 
the simple symmetric exclusion process in the regimes of fast or slow jumps. While their proofs differ 
significantly from ours, their methods are in spirit close. It remains an interesting question to determine 
how far both methods can be pushed (see also Section~\ref{ss:regextensions}).

\medskip\noindent
{\bf 7.}
After the preprint version of the present article appeared, it was brought to our attention that a 
regeneration argument similar to the one used to prove Theorem~\ref{thm:SLLN+CLT+LD} 
appears in \cite{BR}, where the authors consider a different model in the same random 
environment (in continuous time). We believe, however, that our approach is simpler and is 
better suited to the ballisticity condition \eqref{e:LD} (which is mildly stronger than the one 
used in \cite{BR}).

\medskip\noindent 
{\bf Organization of the paper.} 
In Section~\ref{s:preliminaries} we give a graphical construction of our random walk in
dynamic random environment. This construction will be convenient during the proofs of
our main results. In Section~\ref{s:renorm} we set up a renormalisation scheme, and
use this to show that, for large densities of the particles, the random walk moves with a
positive lower speed to the right. This lower speed of the random walk plays the role of
a ballisticity condition and is crucial in Section~\ref{s:regeneration}, where we introduce
a random sequence of regeneration times at which the random walk ``refreshes its outlook
on the random environment'', and show that these regeneration times have a good tail. 
In Section~\ref{ss:limittheorems} the regeneration times are used to prove
Theorem~\ref{thm:SLLN+CLT+LD}. Appendices~\ref{s:simwithPoisson}--\ref{s:conv_rate}
collect a few technical facts that are needed along the way.

\medskip\noindent 
{\bf Acknowledgments.} 
VS thanks O.\ Angel, V.\ Beffara and G.\ Kozma for fruitful discussions at the initial stages 
of the work, and Y.\ Peres and O.\ Zeitouni for discussions on regeneration. MH was supported 
by the Dutch mathematics cluster STAR during an extended visit to the Centre for Mathematics 
and Computer Science in Amsterdam in the Fall of 2010, by CNPq grants 248718/2013-4 and by ERC AG "COMPASP". 
FdH is supported by ERC Advanced Grant 267356 VARIS, 
RdS by the French ANR project MEMEMO2 10-BLAN-0125-03 and the German DFG project KO 2205/13-1, 
VS by Brazilian CNPq grants 308787/2011-0 and 476756/2012-0 and FAPERJ grant E-26/102.878/2012-BBP, 
and AT by CNPq grants 306348/2012-8 and 478577/2012-5. MH thanks Microsoft Research and IMPA
for hospitality and financial support. 
RdS thanks IMPA for hospitality and financial support.
VS thanks Microsoft Research, UC Berkeley, MSRI, FIM ETH-Zurich 
and ENS-Paris for hospitality and financial support. 
This work was also supported by the ESF RGLIS Excellence Network.


\section{Preliminaries}
\label{s:preliminaries}

In this section we give a particular construction of our model, supporting a Poisson point process 
on the space of two-sided trajectories of environment particles (Section~\ref{ss:DRE}) and an i.i.d.\ 
sequence of $\Uniform([0,1])$ random variables that are used to define our random walk
(Section~\ref{ss:RWDRE}). This formulation is equivalent to that given in Section~\ref{s:intro},
but has the advantage of providing independence and monotonicity properties that are useful
throughout the paper (see Definitions~\ref{d:monotone}--\ref{d:support} and
Remark~\ref{r:monotone} below).

Throughout the sequel, $c$ denotes a positive constant that may depend on $v_\circ, v_\bullet$
and may change each time it appears. Further dependence will be made explicit: for example,
$c(\eta)$ is a constant that depends on $\eta$ and possibly on $v_\circ, v_\bullet$. Numbered
constants $c_0, c_1, \dots$ refer to their first appearance in the text and also depend only on 
$v_\circ$ and $v_\bullet$ unless otherwise indicated.


\subsection{Dynamic random environment}
\label{ss:DRE}

Let
\begin{equation}
\label{e:defenvpart}
S = (S^{z,i})_{i \in \N, z \in \Z} \;\; \text{ with } \;\; S^{z,i} = (S^{z,i}_n)_{n\in \Z}
\end{equation}
be a doubly-indexed collection of independent lazy simple random walks such that $S^{z,i}_0 = z$
for all $i \in \N$. By this we mean that the past $(S_n^{z,i})_{n \in \mathbb{Z}_-}$ and the future 
$(S_n^{z,i})_{n \in \mathbb{Z}_+}$ are independent and distributed as symmetric lazy simple random 
walks as described in Section~\ref{s:intro}.

Let $(N(z,0))_{z \in \Z}$ be a sequence of i.i.d.\ random variables independent of $S$. Then, the
process $N(\cdot,n)$ defined by
\begin{equation}
\label{e:defIPS}
N(x,n) = \sum_{z \in \N} \, \sum_{1 \leq i \leq N(z,0)} \mathbbm{1}_{\{S^{z,i}_n=x\}},
\qquad (x, n) \in \Z^2,
\end{equation}
(with $0$ assigned to empty sums) is a translation-invariant Markov process representing the number
of environment particles at site $x$ and time $n$. For any density $\rho >0$, the process $N$ is in equilibrium
when we choose the distribution of $N(\cdot,0)$ to be product Poisson($\rho$). Denote by $\P^\rho$
the joint law of $N(\cdot,0)$ and $S$ in this case.

It will be useful to view $N$ as a subprocess of a Poisson point process on a space of trajectories as
follows.
Let
\begin{equation}
\label{e:W}
W = \big\{ w=(w(n))_{n \in \Z}\colon\, w(n) \in \Z,\,|w(n+1) - w({n})| \leq 1\,\,\forall\,n \in \Z \big\},
\end{equation}
denote the set of two-sided nearest-neighbour trajectories on $\Z$. Endow $W$ with the sigma-algebra
$\cW$ generated by the canonical projections $Z_n(w) = w(n)$, $n \in \Z$. A partition of $W$ into disjoint
measurable sets is given by $\{W_x\}_{x \in \Z}$, where $W_x = \{w \in W\colon\, w(0) = x\}$.

We introduce the space $\bar{\Omega}$ of point measures on $W$ as (be careful to distinguish
$\omega$ from $w$)
\begin{equation}
\label{e:Omegabar}
\bar{\Omega} = \Big\{ \omega = \sum_{j \in \Z_+} \delta_{w_j}\colon\, w_j \in W
\,\,\forall\, j \in \Z_+, \, |\omega(W_x)| < \infty \,\,\forall\, x \in \Z \Big\},
\end{equation}
and define a random point measure $\omega \in \bar{\Omega}$ by the expression
\begin{equation}
\label{e:randomPPP}
\omega = \sum_{z \in \Z} \, \sum_{1 \leq i \leq N(z,0)} \delta_{S^{z,i}}.
\end{equation}
It is then straightforward to check that, under $\P^\rho$, $\omega$ is a Poisson point
process on $W$ with intensity measure $\rho \mu$, where
\begin{equation}
\label{e:mu}
\mu = \sum_{x \in \Z} P_x
\end{equation}
and $P_x$ is the law on $W$, with support on $W_x$, under which $Z(\cdot) =
(Z_n(\cdot))_{n \in \Z}$ is distributed as a lazy simple random walk on $\Z$.
Moreover, we have that
\begin{equation}
\label{e:relNomega}
N(x,n) = \omega(\{w \in W \colon w(n) = x\}), \qquad (x,n) \in \Z^2.
\end{equation}

For $w \in W$, let $\Trace(w) =\{(w(n),n)\}_{n \in \Z} \subset \Z^2$ be the trace of $w$,
and define the total trace of the environment trajectories as the set
\begin{equation}
\label{e:T_rho}
\cT = \cT(\omega) = \bigcup_{z \in \Z} \bigcup_{1 \le i \le N(z,0)} \Trace(S^{z,i}).
\end{equation}

We can now justify \eqref{e:cov}. Noting that $\{(x,k) \notin \cT\} = \{\omega(w(k)=x)=0\}$, compute 
\begin{equation}\label{e:justcov1}
\begin{aligned}
\P^\rho \big[ (0,0) \notin \cT, (0,n) \notin \cT \big] & = \P^{\rho} \left[ \omega\big(w(0)=0 \text{ or } w(n)=0 \big)=0\right] \\
& = \exp \left\{- \rho \mu\big(w(0)=0 \text{ or } w(n) = 0 \big) \right\}.
\end{aligned}
\end{equation}
Now,
\begin{equation}\label{e:justcov2}
\begin{aligned}
\mu\big(w(0)=0  \text{ or } w(n) = 0\big) & = \mu\big(w(0)=0\big) + \mu\big(w(n)=0\neq w(0) \big) \\
& = 1+P_0(Z_n\neq0) = 2 - P_0(Z_n=0),
\end{aligned}
\end{equation}
where we used the symmetry of $Z$.
Hence,
\begin{align}\label{e:justcov3}
 \Cov_\rho \left( \mathbbm{1}_{\{(0,0) \in \cT\}}, \mathbbm{1}_{\{ (0,n) \in \cT \}} \right)
& =\Cov_\rho \left( \mathbbm{1}_{\{(0,0) \notin \cT\}}, \mathbbm{1}_{\{ (0,n) \notin \cT \}}\right) \nonumber\\
& = e^{-2\rho} \left( e^{\rho P_0(Z_n=0)} - 1\right) \sim c(\rho) n^{-\frac12}
\end{align}
since $P_0(Z_n=0) \sim c n^{-\frac12}$.


\subsection{Random walk in dynamic random environment}
\label{ss:RWDRE}

In order to define the random walk on our dynamic random environment, we first
enlarge the probability space. To that end, let us consider a collection of i.i.d.\
random variables $U = (U_y)_{y \in \mathbb{Z}^2}$, independent of the previous
objects, with each $U_y$ uniformly distributed on $[0,1]$. Set $\Omega = \bar{\Omega} 
\times [0,1]^{\Z^2}$, and redefine $\P^\rho$ to be the probability measure giving the 
joint law of $N(\cdot,0)$, $S$ and $U$.

Given a realisation of $\omega$ and $U$ and $y \in \Z^2$, define the random
variables $Y^{y}_n$, $n \in {\mathbb{Z}_+}$, as follows:
\begin{equation}
\label{e:y}
\begin{split}
Y^{y}_0 & = y,\\
Y^{y}_{n+1} & =
\begin{cases}
Y^{y}_n + (1,1), \quad & \text{if } Y^{y}_n \in \cT(\omega)
\text{ and }\,U_{Y^{y}_n} \leq p_\bullet\\
&\text{or } Y^{y}_n \not \in\cT(\omega) \text{ and }\, U_{Y^{y}_n} \leq p_\circ,\\
Y^{y}_n + (-1,1), & \text{otherwise.}
\end{cases}
\end{split}
\end{equation}
In words, $Y^{y} = (Y^{y}_n)_{n\in{\mathbb{Z}_+}}$ is the space-time process on $\Z^2$
that starts at $y$, always moves upwards, and is such that its horizontal projection $X^{y}
= (X^{y}_n)_{n\in{\mathbb{Z}_+}}$ is a random walk with drift $v_{\bullet}=2p_\bullet-1$
when $Y^{y}$ steps on $\cT(\omega)$ and drift $v_{\circ}=2p_\circ-1$ otherwise. Note that
$Y^{y}$ depends on $\cT(\omega)$, but this will be suppressed from the notation. 
Also note that, for any $y \in \Z^2$, the law of $X^y$ under $\mathbb{P}^{\rho}$ coincides with the annealed law described
in Section~\ref{s:intro}. 
So from now on $X=X^0$ will be the random walk in dynamic
random environment that we will consider. 
We may also write $Y_n$ to denote $Y^{0}_n$.

\begin{definition}
\label{d:monotone}
For $\omega,\omega' \in \bar{\Omega}$, we say that $\omega \leq \omega'$ when $\mathcal{T}
(\omega) \subset \mathcal{T}(\omega')$. We say that a random variable $f\colon\,\Omega
\to \mathbb{R}$ is non-increasing when $f(\omega', \xi) \leq f(\omega, \xi)$ for all $\omega
\leq \omega'$ and all $\xi \in [0,1]^{\Z^2}$.
We extend this definition to events $A$ in $\Omega$ by considering $f=\mathbf{1}_A$. 
Standard coupling arguments imply that
$\mathbb{E}^{\rho'}(f) \leq \mathbb{E}^{\rho}(f)$ for all non-increasing random variables
$f$ and all $\rho \leq \rho'$.
\end{definition}

\begin{definition}
\label{d:support}
We say that a random variable $f\colon\, \Omega \to \mathbb{R}$ has support in $B
\subset \mathbb{Z}^2$ when $f(\omega,\xi) = f(\omega',\xi')$ for all $\omega',\omega
\in\bar{\Omega}$ with $\mathcal{T}(\omega) \cap B = \mathcal{T}(\omega') \cap B$
and all $\xi, \xi' \in [0,1]^{\Z^2}$ with $\xi(v) = \xi'(v)$ for all $v \in B$.
\end{definition}

\begin{remark}
\label{r:monotone}
The above construction provides two forms of monotonicity:\\
(i) Initial position: If $x \leq x'$ have the same parity (i.e., $x' - x \in 2 \mathbb{Z}$),
then
\begin{equation}
X^{(x,n)}_i \leq X^{(x',n)}_i \qquad \forall\, n \in \mathbb{Z}\,\,\forall\, i \in \mathbb{Z}_+.
\end{equation}
(ii) Environment: If $v_\circ \leq v_\bullet$ and $\omega \leq \omega'$, then
\begin{equation}
X^{y}_i(\omega) \leq X^y_i(\omega') \qquad \forall\, y \in \Z^2 \,\,\forall\,i \in \mathbb{Z}_+.
\end{equation}
\end{remark}


\section{Proof of Theorem~\ref{thm:LD}: Renormalisation}
\label{s:renorm}

In this section we prove Theorem~\ref{thm:LD}, which shows the validity of the ballisticity
condition in \eqref{e:LD} when $v_\circ < v_\bullet$ and $\rho$ is large enough. 
This will be crucial for the proof of Theorem~\ref{thm:SLLN+CLT+LD} later.

In Section~\ref{ss:notation} we introduce the required notation. In Section~\ref{ss:pkest} we
devise a \emph{renormalisation scheme} (Lemmas~\ref{l:recursion}--\ref{l:induction}) to
show that under a ``finite-size criterion'' the random walk moves ballistically, and we prove
that for large enough $\rho$ this criterion holds (Lemma~\ref{l:trigger}). In Section~\ref{ss:ldest}
we show that the renormalisation scheme yields the large deviation bound in
Theorem~\ref{thm:LD} (Lemma~\ref{l:hit_line}). This bound will be needed in
Section~\ref{s:regeneration}, where we show that, as the random walk explores fresh parts
of the dynamic random environment, it builds up a \emph{regeneration structure} that serves
as a ``skeleton'' for the proof of Theorem~\ref{thm:SLLN+CLT+LD}.
In Section~\ref{ss:commentsreg} we comment on possible extensions.


\subsection{Space-time decoupling}
\label{ss:space_time}

In order to implement our renormalisation scheme, we need to control the dependence of
events having support in two boxes that are well separated in space-time. This is the content
of the following corollary of Theorem~\ref{t:decouple}, the proof of which is deferred to
Appendix~\ref{s:decouple}.

\begin{corollary}
\label{c:decouple}
Let $B_1 = ([a, b] \times [n, m]) \cap \Z^2$ and $B_2 = ([a', b'] \times [-n', 0]) \cap \Z^2$ be two
space-time boxes (observe that their time separation is $n$) and assume that $n \geq c$.
Recall Definitions~\ref{d:support} and \ref{d:monotone}, and assume that $f_1\colon\,\Omega
\to [0,1]$ and $f_2\colon\,\Omega \to [0,1]$ are non-increasing random variables with support
in $B_1$ and $B_2$, respectively. Then, for any $\rho \geq 1$,
\begin{equation}
\label{e:dec2}
\mathbb{E}^{\rho(1+n^{-1/16})}[f_1 f_2] \leq \mathbb{E}^{\rho(1+n^{-1/16})}[f_1] \,\,
\mathbb{E}^{\rho}[f_2] + c \big(\textnormal{per}(B_1) + n\big)\,e^{-c n^{1/8}},
\end{equation}
where $\textnormal{per}(B_1)$ stands for the perimeter of $B_1$.
\end{corollary}

The decoupling in Corollary~\ref{c:decouple}, together with the monotonicity stated in
Definition~\ref{d:monotone}, are the only assumptions on our dynamic random environment
that are used in the proof of Theorem~\ref{thm:LD}. Hence, the results in this section can in
principle be extended to different dynamic random environments. (See Section~\ref{ss:commentsreg}
for more details.)

\subsection{Scale notation}
\label{ss:notation}

\newconstant{c:Lk_growth}
Define recursively a sequence of scales $(L_k)_{k\in{\mathbb{Z}_+}}$ by putting
\begin{equation}
\label{e:L_k+1}
L_0 = 100, \qquad L_{k+1} = \lfloor L_k^{1/2} \rfloor L_k.
\end{equation}
(The choice $L_0=100$ has no special importance: any integer $\geq 4$ will do, as long as
it stays fixed.) Note that the above sequence grows super-exponentially fast:  $\log L_k \sim
(3/2)^k \log L_0$ as $k\to\infty$. For $L\in\N$, let $B_L$ be the space-time rectangle
\begin{equation}
\label{e:box}
B_L = \big([-L, 2L] \times [0,L]\big) \cap \Z^2
\end{equation}
and $I_L$ its middle bottom line
\begin{equation}
\label{e:I_L}
I_L = [0,L]\times \{0\} \subset B_L
\end{equation}
(see Fig.~\ref{fig-rect}). For $m = (r,s) \in \Z^2$, let
\begin{equation}
\label{e:box_trans}
B_L (m) = B_L(r,s) = (r,s)L + B_L, \qquad I_L(m) = I_L(r,s) = (r,s)L + I_L.
\end{equation}
For $k \in \N$, let
\begin{equation}
\label{e:M_K}
M_k = \big\{(r,s) \in \Z^2\colon\, B_{L_k} (r,s) \cap B_{L_{k+1}} \neq \emptyset\big\}
\end{equation}
denote the set of all indices whose corresponding shift of the rectangle $B_{L_k}$ still
intersects the larger rectangle $B_{L_{k+1}} = B_{L_{k + 1}}(0,0)$.

\begin{figure}[htbp]
\vspace{-1.7cm}
\begin{center}
\setlength{\unitlength}{0.5cm}
\begin{picture}(12,10)(0,-2)
{\thicklines
\qbezier(0,0)(7.5,0)(15,0)
\qbezier(0,4)(7.5,4)(15,4)
\qbezier(0,0)(0,2)(0,4)
\qbezier(15,0)(15,2)(15,4)
\qbezier(5,.05)(7.5,.05)(10,.05)
\qbezier(5,-.05)(7.5,-.05)(10,-.05)
}
\put(-.2,-1.5){$-L$}
\put(4.8,-1.5){$0$}
\put(9.8,-1.5){$L$}
\put(14.6,-1.5){$2L$}
\put(-1,0){$0$}
\put(-1,4){$L$}
\put(5,0){\circle*{.35}}
\put(10,0){\circle*{.35}}
\end{picture}
\end{center}
\caption{\small Picture of $B_L$ (rectangle) and $I_L$ (middle bottom line).}
\label{fig-rect}
\end{figure}
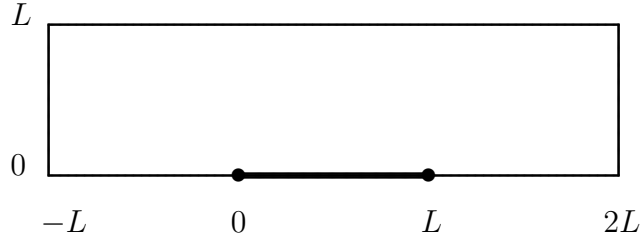

Fix $v <  v_{\bullet}$, let $\delta = \tfrac12(v_\bullet-v)$, and define recursively a sequence
$(v_k)_{k\in\N}$ of velocities by putting
\begin{equation}
\label{e:seq_vel}
v_1 = v_{\bullet} - \delta, \qquad v_{k+1} = v_k - \delta \left(\frac{6}{\pi^2} \right) \frac{1}{k^2} .
\end{equation}
Since $\sum_{k\in\N} 1/k^2 = \pi^2/6$, it follows that $k \mapsto v_k$ decreases strictly to $v$.
The reason why we introduce a speed for each scale $k$ is to allow for small errors as we
change scales. (The need for this ``perturbation'' will become clear in \eqref{e:displace} below.)

We are interested in bounding the probability of bad events $A_k$ on which the random walk 
does not move to the right with speed at least $v_k$, namely,
\begin{equation}
\label{e:bad_event}
A_k(m) = \Big\{ \exists\, (x,n) \in I_{L_k}(m)\colon\, X^{(x,n)}_{L_k} - x < v_k L_k \Big\},
\qquad k \in \N, m \in \Z^2.
\end{equation}
Note that $A_k(m)$ is defined in terms of the dynamic random environment and the random
walk within $B_{L_k}(m)$, so that $\mathbbm{1}_{A_k(m)}$ is a random variable with support
in $B_{L_k}(m)$, in the sense of Definition~\ref{d:support}. We say that $B_{L_k}(m)$ is a
\emph{slow box} when $A_k(m)$ occurs. Since we are assuming that $v_{\bullet} > v_{\circ}$
(recall \eqref{r:v_leq_v}), for each $k$ and $m$ the random variable $\mathbbm{1}_{A_k(m)}$
is non-increasing in the sense of Definition~\ref{d:monotone}.

Define recursively a sequence $(\rho_k)_{k\in{\mathbb{Z}_+}}$ of densities by putting
\begin{equation}
\label{e:density}
\rho_0 > 0, \qquad \rho_{k+1} = (1+ L_{k}^{-1/16})\rho_{k}.
\end{equation}
Again, we introduce a density for each scale $k$ in order to allow for small errors. (The
need for this ``sprinkling" will become clear in \eqref{e:use_decouple} below.)
Observe that $\rho_k$ increases strictly to $\rho_\star$ defined by
\begin{equation}
\label{defrhostar}
\rho_\star = \rho_0 \prod_{l = 0}^{\infty} (1 + L_{l}^{-1/16}) \in (\rho_0, \infty).
\end{equation}

Finally, define
\begin{equation}
\label{e:p_k}
p_k = \P^{\rho_k}(A_k(0)) = \mathbb{P}^{\rho_k}(A_k(m)), \qquad k \in \N,
\end{equation}
where the last equality holds for all $m \in \Z^2$ because of translation invariance.

\subsection{Estimates on $p_k$}
\label{ss:pkest}

Lemmas~\ref{l:recursion}--\ref{l:trigger} below show that $p_k$ decays very rapidly with $k$
when $\rho_0$ is chosen large enough.

The first step is to prove a recursion inequality that relates $p_{k+1}$ with $p_k$:

\newconstant{c:recursion0}
\newconstant{c:recursion}
\begin{lemma}
\label{l:recursion}
Fix $\rho_0 \geq 1$. There is a $k_0 = k_0(\delta)$ such that, for all $k > k_0(\delta)$,
\begin{equation}
\label{e:recursion}
p_{k+1} \leq \useconstant{c:recursion0} L_k^2
\left[p_k^2 + L_k e^{- \useconstant{c:recursion} L_k^{1/8}}\right].
\end{equation}
\end{lemma}

\begin{proof}
Let $k_0 = k_0 (\delta)$ be a non-negative integer such that, for all
$k \geq k_0(\delta)$,
\begin{equation}
\label{e:k2large}
\delta \left(\frac{6}{\pi^2} \right)\frac{1}{k^2}\geq
\frac {4}{\lfloor L_k^{1/2} \rfloor}.
\end{equation}
The existence of $k_0$ follows from the fact that $L_k$ increases faster than exponentially in $k$.
We begin by claiming the following:
\begin{equation}
\label{e:A_k+1_k}
\begin{array}{ll}
&\text{For all $k \geq k_0$, if $A_{k+1}(0)$ occurs, then there are at least}\\
&\text{three elements $m_1 = (r_1, s_1)$, $m_2 = (r_2, s_2)$, $m_3 = (r_3, s_3)$ in $M_k$,}\\
&\text{with $s_i \neq s_j$ when $i \neq j$, such that $A_k (m_i)$ occurs for $i = 1,2,3$.}
\end{array}
\end{equation}
The proof is by contradiction. Suppose that the claim is false. Then there are at most
two elements $m = (r,s)$, $m' = (r', s')$ in $M_k$, with $s \neq s'$, such that $B_{L_k}(m)$
and $B_{L_k}(m')$ are slow boxes. It then follows that, for any $(x,n) \in I_{L_{k+1}}$,
\begin{equation}
\label{e:displace}
\begin{array}{rcl}
X^{(x,n)}_{L_{k+1}} -x & = & \sum_{j = 1}^{ \lfloor L_k^{1/2} \rfloor}
\big[X^{(x,n)}_{j L_k} - X^{(x,n)}_{(j-1) L_k}\big]\\[0.2cm]
& \geq & -2 L_{k} + v_k L_k \left(\frac{L_{k+1}}{L_k} - 2\right) \\ [0.2cm]
& \geq & -4L_k + v_k L_{k+1},
\end{array}
\end{equation}
where the terms in the sum correspond to the displacements over the $\lfloor L_k^{1/2} \rfloor$
time layers of height $L_k$ in the box $B_{L_{k+1}}$. The term $-2L_k$ appears in the right-hand
side of the first inequality because there are at most two layers (associated with the two slow
boxes mentioned above) in which the total displacement of the random walk is at least $-L_k$,
since the minimum speed is $-1$. The second inequality uses that $v_k \leq 1$. 

By the definition of $k_0(\delta)$ we get,
\begin{equation}
\begin{array}{e}
-4L_k + v_k L_{k+1}
&=& -4 L_k + \left[ \delta \left(\tfrac{6}{\pi^2} \right)\tfrac{1}{k^2}  \right]
L_{k+1} + v_{k+1} L_{k+1}\\[3mm]
&=& - \tfrac {4}{\lfloor L_k^{1/2} \rfloor} L_{k+1}
+ \left[ \delta \left(\tfrac{6}{\pi^2} \right)\tfrac{1}{k^2}  \right] L_{k+1}+ v_{k+1} L_{k+1}\\[3mm]
&\geq & v_{k+1} L_{k+1}.
\end{array}
\end{equation}
Substituting this into \eqref{e:displace} we get
\begin{equation}
X^{(x,n)}_{L_{k+1}} - x \geq v_{k+1} L_{k+1} \qquad \forall\, (x,n) \in I_{L_{k+1}},
\end{equation}
so that $A_{k+1}(0)$ cannot occur. This proves the claim \eqref{e:A_k+1_k}.

Thus, on the event $A_{k+1}(0)$, we may assume that there exist $m_1 = (r_1, s_1)$,
$m_3 = (r_3, s_3)$ in $M_k$ such that $s_3 \geq s_1 + 2$, meaning that the vertical
distance between $B_{L_k}(m_3)$ and $B_{L_k}(m_1)$ is at least $L_k$. It follows
from Corollary~\ref{c:decouple} and the fact that the events $A_k(m)$ are non-increasing
that
\begin{equation}
\label{e:use_decouple}
\begin{array}{rcl}
\mathbb{P}^{\rho_{k+1}} \left( A_k (m_1) \cap A_k (m_2) \right)
& \leq & \mathbb{P}^{\rho_{k+1}} \left( A_k (m_1) \right)
\mathbb{P}^{\rho_k} \left( A_k(m_3) \right) \\[0.2cm]
&  & \qquad +  c\, [\text{per}(B_{L_k}) + L_k]\, e^{- c \rho_{k} L_k^{1/8}} \\[0.2cm]
&\leq & \mathbb{P}^{\rho_{k}} \left( A_k (m_1) \right)^2
+  c\, [\text{per}(B_{L_k}) + L_k]\, e^{- c \rho_{k} L_k^{1/8}} \\[0.2cm]
&\leq & p_k^2 + c L_k e^{ - c L_k^{1/8}},
  \end{array}
\end{equation}
where $\text{per}(B_{L_k})$ denotes the perimeter of $B_{L_k}$, and in the last inequality
we use that $\rho_k \geq \rho_0 \geq 1$.
Since there are at most $c (L_{k+1}/L_k)^4 = c \lfloor
L_k^{1/2} \rfloor^4$ possible choices of pairs of boxes $B_{L_k}(m_1)$ and $B_{L_k}(m_3)$
in $M_k$, it follows that
\begin{equation}
p_{k+1} \leq c L_k^2 \left[p_k^2 + L_k e^{-c L_k^{1/8}}\right],
\end{equation}
which completes the proof of \eqref{e:recursion}.
\end{proof}

Next, we prove a recursive estimate on $p_k$.

\begin{lemma}
\label{l:induction}
There exists a $k_1 = k_1(\delta) \geq k_0(\delta)$ such that, for all $k \geq k_1$,
\begin{equation}
\label{e:recstep}
p_k < e^{-\log^{3/2} L_k} \quad \Longrightarrow \quad p_{k+1} < e^{-\log^{3/2} L_{k+1}}.
\end{equation}
\end{lemma}

\begin{proof}
Suppose that $p_k < e^{-\log^{3/2} L_k}$ for some $k \geq k_0 (\delta)$. For such a $k$,
Lemma \ref{l:recursion} gives
\begin{equation}
\label{e:recursion_2}
p_{k+1} \leq \useconstant{c:recursion0} L_k^2
\left[p_k^2 + L_k e^{-\useconstant{c:recursion} L_k^{1/8}}\right]
\leq \useconstant{c:recursion0} L_k^2 \left[e^{-2\log^{3/2} L_k} +
L_k e^{-\useconstant{c:recursion} L_k^{1/8}}\right].
\end{equation}
Pick $k_1 = k_1(\delta) \geq k_0 (\delta)$ such that
\begin{equation}
\label{e:recursion_3}
\useconstant{c:recursion0} L_k^2 \left(e^{-(1/10) \log^{3/2} L_k}
+ L_k e^{- \useconstant{c:recursion} L_k^{1/8} + \log^{3/2} L_{k+1}} \right)
< 1 \qquad  \forall\, k \geq k_1,
\end{equation}
which is possible because  $\lim_{k\to\infty} L_k = \infty$. Dividing \eqref{e:recursion_2}
by $e^{-\log^{3/2} L_{k+1}}$, recalling from \eqref{e:L_k+1} that $L_{k+1} \leq L_k^{3/2}$
and using \eqref{e:recursion_3}, we get
\begin{equation}
\begin{split}
p_{k+1}\,e^{\log^{3/2} L_{k+1}}
& \leq \useconstant{c:recursion0} L_k^2
\Big[e^{ -\overbrace{\scriptstyle(2 - (3/2)^{3/2})}^{> (1/10)} \log^{3/2} L_{k}}
+ L_k e^{-\useconstant{c:recursion} L_k^{1/8} + \log^{3/2} L_{k+1}}\Big]
\overset{\eqref{e:recursion_3}} < 1,
\end{split}
\end{equation}
which completes the proof of the \eqref{e:recstep}.
\end{proof}

Finally, we show that if $\rho_0$ is taken large enough, then it is possible to trigger the
recursive estimate in \ref{l:induction}.

\begin{lemma}
\label{l:trigger}
There exist $\rho_0$ large enough and $k_2 = k_2(\delta) \geq k_1(\delta)$ such that
$p_{k_2}<e^{-\log^{3/2} L_{k_2}}$.
\end{lemma}

\begin{proof}
Recall \eqref{e:box}, \eqref{e:bad_event} and \eqref{e:p_k}. Recall also from \eqref{e:defIPS}
that $N(x,n)$ denotes the number of particles in our dynamic random environment that cross
$(x,n)$ (i.e., $N(x,n) = \omega(\{w \in W; w(n) = x\})$), and let
\begin{equation}
C_k = \Big\{N(x,n) \geq 1 \,\,\forall\,(x,n) \in B_{L_k} \Big\}
\end{equation}
be the event that all space-time points in $B_{L_k}$ are occupied by a particle.
Estimate
\begin{equation}
\label{e:p_k A_k C_k}
p_k \leq \mathbb{P}^{\rho_k}( A_k(0) \mid C_k) + \mathbb{P}^{\rho_k} (C_k^c).
\end{equation}
The first term in the right-hand side of \eqref{e:p_k A_k C_k} can be estimated from above by
\begin{equation}
\label{e:A_k|C_k}
\begin{array}{rcl}
\mathbb{P}^{\rho_k}(A_k(0) \mid C_k)
&\leq& L_k \mathbb{P}^{\rho_k} (X^0_{L_k} < v_k L_k \mid C_k) \\[0.2cm]
&\leq& L_k \mathbb{P}^{\rho_k} ( X_{L_k}^0 < v_1 L_k \mid C_k)\\[0.2cm]
&=& L_k \mathbb{P}^{\rho_k} \Big(\frac{X_{L_k}^0}{L_k}  < v_\bullet - \delta \mid C_k\Big),
\end{array}
\end{equation}
where the last inequality uses \eqref{e:seq_vel}. On the event $C_k$, all the space-time points
of $B_{L_k}$ in the dynamic random environment are occupied, and so the law of
$(X_n^{0})_{0 \leq n \le L_k}$ coincides with that of a nearest-neighbour random walk
with drift $v_{\bullet}$ starting at $0$. Therefore, by an elementary large deviation estimate,
we have
\begin{equation}
\mathbb{P}^{\rho_k}(A_k(0) \mid C_k) \leq  L_k e^{-c(\delta) L_k}, \qquad k \in \N,
\end{equation}
independently of the choice of $\rho_0$. We can therefore choose $k_2 = k_2(\delta)$ large
enough so that
\begin{equation}
\label{e:A_k_2|C_k_2}
\mathbb{P}^{\rho_k}(A_{k_2}(0) \mid C_{k_2}) \leq \tfrac12 e^{-\log^{3/2} L_{k_2}}.
\end{equation}

Having fixed $k_2$, we next turn our attention to the second term in the right-hand side of
\eqref{e:p_k A_k C_k}. Recalling that, under $\mathbb{P}^{\rho_{k_2}}$, the random variables
$(N(x,n))_{x \in \mathbb{Z}}$ are $\Poisson(\rho_{k_2})$, we have $\mathbb{P}^{\rho_{k_2}}
(C_{k_2}^c) \leq 3 L^2_{k_2} e^{-\rho_{k_2}}$. Since this tends to zero as $\rho_0\to\infty$
(recall \eqref{e:density}), we can take $\rho_0$ large enough so that
\begin{equation}
\label{e:p_kC_k_2c}
\mathbb{P}^{\rho_{k_2}}(C_{k_2}^c) \leq \tfrac12 e^{-\log^{3/2} L_{k_2}}.
\end{equation}
Combine \eqref{e:p_k A_k C_k}, \eqref{e:A_k_2|C_k_2} and \eqref{e:p_kC_k_2c} to get the claim.
\end{proof}

\subsection{Large deviation bounds}
\label{ss:ldest}

Together with Lemmas~\ref{l:induction}--\ref{l:trigger}, the following lemma will allow us to
prove Theorem~\ref{thm:LD}.

Define the half-plane
\begin{equation}
\label{e:cal_H}
\mathcal{H}_{v,L} = \big\{ (x,n) \in \mathbb{Z}^2\colon\,x \leq n v - L  \big\}.
\end{equation}

\newconstant{c:k_hit_line}
\begin{lemma}
\label{l:hit_line}
Fix $v < v_\bullet$ and $\rho > 0$. Suppose that $\mathbb{P}^\rho(A_k) \leq e^{-\log^{3/2} L_k}$
for all $k \geq \useconstant{c:k_hit_line}$ (where $A_k$ is defined in terms of $v$ as in
\eqref{e:bad_event}). Then, for any $\varepsilon > 0$,
\begin{equation}
\label{e:hit_line}
\mathbb{P}^\rho \Big(\exists\,n\in\N\colon\, Y_n \in \mathcal{H}_{v - \varepsilon,L} \Big)
\leq c (\varepsilon, \useconstant{c:k_hit_line}) L^{7/2} \,e^{-c\log^{3/2} L}
\qquad \forall\, L \in\N.
\end{equation}
\end{lemma}

\begin{proof}
We first choose $k_0 = k_0(\varepsilon)$ such that $L_{k+1}/L_k > 1 + 2/\varepsilon$ for all
$k \geq k_0$. A trivial observation is that we may assume that $L > 2 L_{k_0 \vee
\useconstant{c:k_hit_line} + 2}$, as this would at most change the constant $c(\varepsilon,
\useconstant{c:k_hit_line})$ in \eqref{e:hit_line}. We thus choose $\check k$ such that
\begin{equation}
  \label{e:between_twos}
  2L_{\check k + 2} \leq L < 2 L_{\check k + 3}.
\end{equation}
Note that $\check k \geq k_0$ by our assumption on $L$.

We next define the set of indices (see Fig.~\ref{f:Mk_bar})
\begin{equation}
M'_k = \{m \in M_k \colon\, B_{L_k}(m) \subseteq B_{L_{k+2}} (0) \},
\qquad k\in{\mathbb{Z}_+},
\end{equation}
and consider the event
\begin{equation}
B_{\check k} = \bigcap_{k \geq \check k} \bigcap_{m \in M'_{k}} A_{k}(m)^c.
\end{equation}
This event has high probability. Indeed, according to our hypothesis on the decay
of $\mathbb{P}^\rho(A_k)$, and since $\check k \geq \useconstant{c:k_hit_line}$,
we have
\begin{equation}
\label{e:PBkc}
\begin{array}{e}
\mathbb{P}^\rho(B_{\check k}^c)
& \leq & \sum_{k \geq \check k} \sum_{m \in M'_{k}} \mathbb{P}^\rho(A_{k}(m))
\leq \sum_{k \geq \check k} c
\Big(\frac{L_{{k} + 2}}{L_{k}}\Big)^2 e^{-\log^{3/2} L_{k}}\\[0.2cm]
& \leq &
c \sum_{l \geq L_{\check k}} l^{5/2} e^{-\log^{3/2} l}
\leq c L_{\check k}^{7/2} e^{ - \log^{3/2} L_{\check k}}
\leq c L^{7/2} e^{- c \log^{3/2} L},
\end{array}
\end{equation}
where in the fourth inequality we use Lemma~\ref{l:ineq}, while in the last inequality we use
that $L_{\check k} > L_{\check k + 3}^{(3/2)^{-3}} \geq cL^{(3/2)^{-3}}$ (see \eqref{e:L_k+1}
and \eqref{e:between_twos}) and that $2L_{\check k} < L$. It is therefore enough to show
that the event in \eqref{e:hit_line} is contained in $B_{\check k}^c$.

\begin{figure}[htbp]
\centering
\begin{tikzpicture}[scale=.1]
\draw (-50, 0) to  (80, 0);
\draw (-60, 0) to  (-60, 54);
\draw [fill] (0,0) circle [radius=.2];
\foreach \y in {2,1,0} {
\pgfmathparse{3^\y}
\xdef\z{\pgfmathresult}
\pgfmathparse{90 - 30*\y}
\xdef\co{\pgfmathresult}
\ifnum \y < 2
\foreach \w in {0,...,8} {
\pgfmathparse{floor(2*\w / 3)}
\xdef\r{\pgfmathresult}
\draw[fill, color=gray!\co!white] (\z*\r - \z, \w*\z) rectangle (\z*\r + 2*\z, \w*\z + \z);
\draw (-60, \w*\z) -- (-60 - 1 - 2*\y, \w*\z);
\foreach \x in {-9,...,17} {
\draw[thin] (\x*\z, \w*\z) rectangle (\x*\z + \z, \w*\z + \z);}}
\else
\foreach \w in {0,...,5} {
\pgfmathparse{floor(2*\w/ 3)}
\xdef\r{\pgfmathresult}
\draw[fill, color=gray!\co!white] (\z*\r - \z, \w*\z) rectangle (\z*\r + 2*\z, \w*\z + \z);
\draw (-60, \w*\z) -- (-60 - 2*\y*\y, \w*\z);
\draw (-60, \w*\z + \z) -- (-60 - 2*\y*\y, \w*\z + \z);
\foreach \x in {-5,...,7} {
\draw[thin] (\x*\z, \w*\z) rectangle (\x*\z + \z, \w*\z + \z);}}
\fi
\draw[ultra thin] (0,0) -- (36, 54);
}
\end{tikzpicture}
\caption{\small Illustration of space-time points $m \in M'_k$, for $k = 0,1,2$. The boxes
$B_k(m)$ for which $I_k(m)$ intersects the line of constant speed $v$ are shaded. The
set $J_0$ is drawn on the left: different scales appear with different tick lengths.}
\label{f:Mk_bar}
\end{figure}
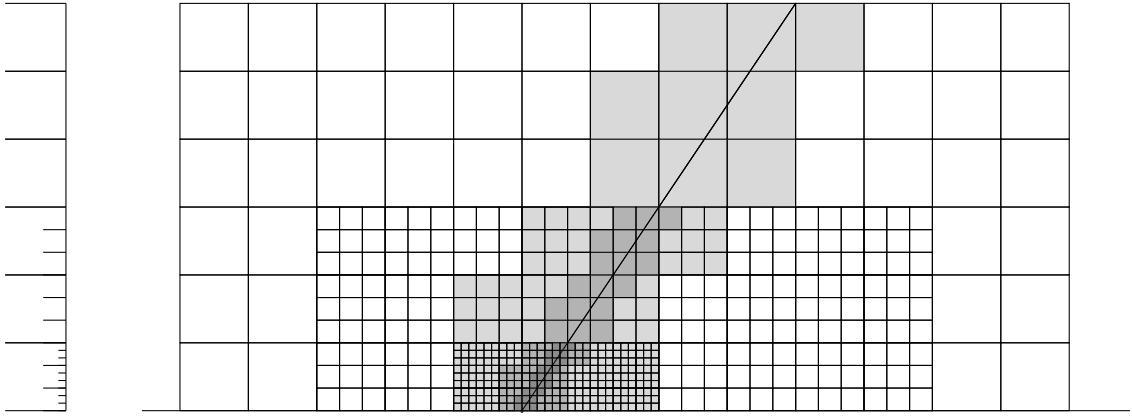

Define the set of times
\begin{equation}
J_{\check k} = \bigcup_{k \geq \check k}
\bigcup_{l = 0}^{L_{{k}+2}/L_{k}} \{ l L_{k} \}  \subset  {\mathbb{Z}_+}.
\end{equation}
We claim that on the event $B_{\check k}$,
\begin{equation}
\label{e:Yj_on_J}
X_j \geq v j \qquad \forall\,j \in J_{\check k}.
\end{equation}
To see why this is true, fix some $k \geq \check k$ as in the definition of $J_{\check k}$. It is clear
that the inequality holds for $j = 0$. Suppose by induction that $X_{l L_k} \geq v l L_{k}$ for some
$l \leq L_{k + 2}/L_k$. Observe that $Y_{l L_k}$ belongs to some box $B_k(m)$ with $m \in
M'_k \subset M'_{\check k}$. It even belongs to the corresponding interval $I_{L_k}(m)$ as defined in
\eqref{e:box_trans}. Since we are on the event $A_k(m)^c$, this implies that
\begin{equation}
X_{(l + 1)L_k} = X_{l L_k} + X^{Y_{l L_k}}_{L_k} - X_{l L_k} \geq v (l + 1) L_k,
\end{equation}
which shows that the bound in \eqref{e:Yj_on_J} holds for $l+1$. Since this can be done for any
$k \geq \check k$, we have proven \eqref{e:Yj_on_J} by induction.

We now interpolate the statement in \eqref{e:Yj_on_J} to all times $n > 2 L_{\check k + 2}
(> L_{\check k + 2} + L_{\check k})$.
More precisely, we will show that, on the event $B_{\check k}$,
\begin{equation}
\label{e:Yt_lower}
X_n \geq (v - \varepsilon) n \qquad \forall \, n \geq 2 L_{\check k + 2}.
\end{equation}
Indeed, given such a $n \geq L_{\check k + 2} + L_{\check k}$, we fix $\bar k$ to be the
smallest $k$ such that
\begin{equation}\label{e:deflbar}
\exists\,l \leq L_{\bar k + 2}/ L_{\bar k}\colon\,n \in \big[ l L_{\bar k}, (l + 1) L_{\bar k} \big),
\end{equation}
and we write $\bar l$ for this unique value of $l$.
Noting that
\begin{equation}
\bar k > \check k \geq k_0 \text{ and } \bar l \ge L_{\bar k + 1}/L_{\bar k} - 1 > 2/\varepsilon,
\end{equation}
we can put the above pieces together and estimate
\begin{align}\label{e:estim_lbar}
X_{n} & = X_{\bar l L_{\bar k}} + X^{Y_{\bar l L_{\bar k}}}_{n - \bar l L_{\bar k}} - X_{\bar l L_{\bar k}}
\geq v \bar l L_{\bar k} - L_{\bar k} = L_{\bar k} (v \bar{l} - 1) \nonumber\\
& = L_{\bar k} \big( (v - \varepsilon) \bar l + \varepsilon \bar l - 1 \big)
\geq L_{\bar k} \big( (v - \varepsilon) \bar l + 1 \big) \nonumber\\
& \geq \max\big(L_{\bar k}( \bar l + 1) (v - \varepsilon), L_{\bar k}{\bar l} (v - \varepsilon)\big) \ge (v - \varepsilon) n,
\end{align}	
where the first inequality uses \eqref{e:Yj_on_J}, $\bar k \geq \check k$ and the
definition of $\bar l$, the second inequality uses that $\bar l > 2/\varepsilon$, 
the third inequality uses that $v-\varepsilon \le 1$ and, for the fourth inequality,
we use \eqref{e:deflbar} considering separately the cases $v-\varepsilon \ge 0$, $v-\varepsilon < 0$.
This proves \eqref{e:Yt_lower}.

To complete the proof, we observe that, since $X$ is Lipschitz, having (as in \eqref{e:Yt_lower})
$X_{n} \geq (v - \varepsilon) n$ for any $n \geq 2 L_{\check k + 2}$ we get $X_{n} \geq
(v - \varepsilon) n - 2 L_{\check k + 2} \geq (v - \varepsilon) n - L$  for all $n \in{\mathbb{Z}_+}$.
Thus, we have proved that the event appearing in the right-hand side of \eqref{e:hit_line} is
contained in $B_{\check k}^c$, so that its probability is bounded as in \eqref{e:PBkc}.
\end{proof}

\begin{proof}[Proof of Theorem~\ref{thm:LD}]
Put $v = v_\star + \varepsilon$, let $\rho_0$ be large enough to satisfy Lemma~\ref{l:trigger},
and take $\rho_\star$ as in \eqref{defrhostar}. Recalling that $X_n$ is the horizontal projection
of $Y_n$ and that, by monotonicity,
\begin{equation}
\P^{\rho_\star}\left(A_k(0)\right) \le p_k \qquad \forall \; k \in \N,
\end{equation}
we see that Lemmas~\ref{l:induction}--\ref{l:hit_line} prove the large deviation bound in
Theorem~\ref{thm:LD}.
\end{proof}

\begin{remark}\label{r:diffLD}
Note that the speed in Lemma~{\rm \ref{l:hit_line}} was not chosen arbitrarily below the speed given
by the law of large numbers in \eqref{e:SLLN}. What we have obtained is that for any $v < v_\bullet$
there exists a density $\rho_0(v)$ such that \eqref{e:hit_line} holds for $\rho \geq \rho_0(v)$.
\end{remark}


\subsection{Extensions}
\label{ss:commentsreg}

The ballisticity statement in Theorem~{\rm \ref{thm:LD}} holds under mild conditions on the
underlying dynamic random environment. Indeed, the only assumptions we have made on
the law of $\mathcal{T}$ are:
\begin{enumerate}[(i)]
\item 
The monotonicity stated in Definition~{\rm \ref{d:monotone}} (see \eqref{e:use_decouple}).
\item 
The decoupling provided by Corollary~{\rm \ref{c:decouple}} (used in \eqref{e:use_decouple}).
\item 
The perturbative condition $\lim_{\rho \to \infty} \mathbb{P}_\rho[0 \in \mathcal{T}] = 1$
(used to trigger \eqref{e:p_kC_k_2c}).
\end{enumerate}
Let us elaborate a bit more on the space-time decoupling condition given by Corollary~{\rm \ref{c:decouple}}.
This condition was designed with our particular dynamic random environment in mind, which lacks
good relaxation properties. However, several dynamic random environments satisfy the simpler
and stronger condition
\begin{equation}
\label{e:fast_mixing}
\mathbb{E}^\rho[f_1 f_2] \leq \mathbb{E}^\rho [f_1] \mathbb{E}^\rho [f_2]
+ c\, \text{per}(B_1)^c e^{-c n^{\kappa}},
\end{equation}
for some $\kappa > 0$ and all $f_1$ and $f_2$ with support in, respectively, $B_1 = [a,b] \times [n, m]$ 
and $B_2 = [a', b'] \times [-n',0]$. It is important to observe that the constants appearing in 
\eqref{e:fast_mixing} are not allowed to depend on $\rho$, since the triggering of \eqref{e:p_kC_k_2c} 
is done after the induction inequality of Lemma~{\rm \ref{l:induction}}. The condition in \eqref{e:fast_mixing} 
holds, for instance, when the dynamic random environment has a spectral gap that is bounded from below 
for $\rho$ large enough. Such a property can be obtained for a variety of reversible dynamics with the help 
of techniques from Liggett~\cite{Li}. 

\medskip\noindent
\textbf{The contact process.}
It can be shown that \eqref{e:fast_mixing} holds for the supercritical contact process for non-increasing 
$f_1$, $f_2$, uniformly in infection parameters that are uniformly bounded away from the critical threshold. 
A proof can be developed using the graphical representation (see e.g.\ Remark 3.7 in 
\cite{dHodSa}) and the strategy of Theorem~\ref{t:decouple}. Note, however, that the results in \cite{dHodSa} 
already imply stronger results for the large deviations of the random walk in the regime of large infection 
parameter, namely, \eqref{e:LD2} with exponential decay.

\medskip\noindent
\textbf{Independent renewal chains.} 
Let us mention another model for which our techniques can establish a ballistic lower bound for the 
random walk. Consider the probability distribution $p = (p_n)_{n \in \mathbb{Z}_+}$ on $\mathbb{Z}_+$ 
given by $p_n = \exp[-n^{1/4}]/Z$, where $Z = \sum_{n \in \mathbb{Z}_+} \exp[-n^{1/4}]$. 
Define the Markov 
chain transition probabilities given by
\begin{equation}
g(l,m) =
\begin{cases}
\delta_{l-1}(m), \quad & \text{if $l \in \N$,}\\
p_m, & \text{if $l = 0$.}
\end{cases}
\end{equation}
This Markov chain moves down one unit a time until it reaches zero. At zero it jumps to a random height
according to distribution $p$. We call this the renewal chain with interarrival distribution $p$. It has stationary 
measure $q = (q_n)_{n \in \mathbb{Z}_+}$ given by
\begin{equation}
q_n = \frac{1}{Z'} \sum_{j \geq n} \exp[-j^{1/4}], 
\qquad Z' = \sum_{n \in \mathbb{Z}_+} \sum_{j \geq n} \exp [-j^{1/4}].
\end{equation}
For each site $x \in \mathbb{Z}$, we produce an independent copy $N(x,n)_{n \in \mathbb{Z}_+}$ of the 
above Markov chain.
Denote by $P_\nu$ the law of one chain started from the probability distribution $\nu$.
We define as a dynamic random environment the field given by these chains when starting from the stationary distribution $q$. 
 
We fix $\rho \geq 0$ and set $\mathcal{T} = \{(x,n)\colon\, N(x,n) 
< \rho \}$, so that we can define the random walk $(Y_n)_{n\in\mathbb{Z}_+}$ as in \eqref{e:y}.

In order to prove Corollary~\ref{c:decouple} for this dynamic random environment,  we would like to couple two renewal chains 
$N(0,n)$, $N'(0,n)$, starting, respectively, at $\delta_0$ and $q$, in such a way that they coalesce at a random time $T$. 
Using Proposition~3 of \cite{Lindvall79}, we obtain such a coupling with $E_{\delta_0,q}[\exp[T^{1/8}]] 
< \infty$ (note that $p$ 
is aperiodic, i.e., $\mathrm{gcd}(\supp(p)) = 1$).

We now fix any events $A \in \sigma(N(0,m) \colon\, m \leq 0)$ and $B \in \sigma(N(0,m) \colon\, m \geq n)$, 
and estimate $P_q[A \cap B] - P_q[A] P_q[B]$. For this, we first check whether $N(0,m)$ reaches zero before 
$n/2$ and, if so, we try to couple it with an independent $N'(0,m)$ starting from the stationary distribution.
This leads to
\begin{equation}
\begin{split}
P_q \big[A\cap B\big] & \leq P_q [N(0,0) > n/2] + P_q \big[A \big] \sup\nolimits_{1 \leq j \leq n/2} P_{\delta_j} \big[ B \big]\\
& \leq P_q [N(0,0)> n/2] + P_{\delta_0,q}[T \geq n/2] + P_q \big[ A \big] P_q \big[ B \big]\\
& \leq P_q \big[A \big] P_q \big[ B \big] + c \exp[- c n^{1/8}],
\end{split}
\end{equation}
where in the last inequality we use the definition of $q$ and the Markov inequality for $\exp[T^{1/8}]$.
Repeating this for every chain $N(x,n)$ with $x \in [a,b]$, we prove \eqref{e:fast_mixing} for $\mathcal{T}$ 
with $\kappa = \tfrac18$. It is clear that $\lim_{\rho\to\infty} P[0 \in \mathcal{T}] = 0$. Thus, the 
conclusion of Theorem~\ref{thm:LD} holds for the dynamic random environment $\mathcal{T}$.

In fact, also Theorem~\ref{thm:SLLN+CLT+LD} holds in this case, as a simple regeneration strategy can be found;
see Section~\ref{ss:regextensions}.
As a consequence, the statements of Theorem~\ref{thm:main} are true for this example.

\begin{remark}
Observe that $\mathcal{T}$ is not uniformly mixing. 
Indeed, given any $n \in \mathbb{Z}_+$, we can start our Markov chain 
in events with positive probability (say, $N(0,0) = 2n$) such that the information at time zero is not forgotten until time $n$.
\end{remark}


\section{Proof of Theorem~\ref{thm:SLLN+CLT+LD}: Regeneration}
\label{s:regeneration}

In this section, we state and prove two results about regeneration times
(Theorems~\ref{t:regeneration}--\ref{t:tailregeneration}) 
that are then used to prove Theorem~\ref{thm:SLLN+CLT+LD} in Section~\ref{ss:limittheorems}.
A discussion about extensions is given in Section~\ref{ss:regextensions}.

In Section~\ref{ss:morenotation} we introduce some additional 
notation in order to define our regeneration time. This definition is made in a non-algorithmic way 
and does not immediately imply that the regeneration time is finite with probability $1$. Nonetheless, 
in the latter event we are able to show in Theorem~\ref{t:regeneration} that a renewal property holds 
for the law of the random walk path. The next step is to prove Theorem~\ref{t:tailregeneration}, which
shows that the regeneration time not only is a.s.\ finite but also has a very good tail. This is accomplished 
by finding a suitable upper bound, which consists of two main steps. First, we define what we call 
\emph{good record times} and show that these appear very frequently (Proposition~\ref{p:manygrts}). 
This is done in an algorithmic fashion, but only by exploring the system locally at each step. Second, 
we show that, outside a global event of small probability, 
if we can find a good record time then we can also find nearby an upper bound for the regeneration time. 


\subsection{Notation}
\label{ss:morenotation}


Suppose that $\rho \in (0,\infty)$, $v_\star \in (0, v_{\bullet})$ and $c\in (0,\infty)$ satisfy \eqref{e:LD}.
Conditions for this are given in Theorem~\ref{thm:LD} and Remark~\ref{r:v_leq_v}. In the
sequel we abbreviate $\P = \P^\rho$.

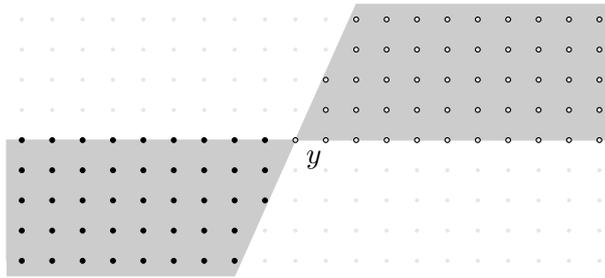
\begin{figure}[htbp]
\centering
\begin{tikzpicture}[scale=.4, font=\small]
\draw[fill, color=gray!40!white] (.5,1) -- (.5,5) -- (10,5) -- (8,.5) -- (.5,.5);
\draw[fill, color=gray!40!white] (10,5) -- (12,9.5) -- (20.5,9.5) -- (20.5,5) -- (10,5);
\foreach \x in {1,...,20} {
\foreach \y in {1,...,9} {
\draw[fill, color=gray!20!white] (\x,\y) circle [radius=0.05];}}
\foreach \y in {1,...,5} {
\pgfmathparse{9 - floor(2*(5-\y)/5)}
\xdef\z{\pgfmathresult}
\foreach \x in {1,...,\z} {
\draw[fill] (\x,\y) circle [radius=0.08];}}
\foreach \y in {5,...,9} {
\pgfmathparse{10 - floor(2*(5-\y)/5)}
\xdef\z{\pgfmathresult}
\foreach \x in {\z,...,20} {
\draw (\x,\y) circle [radius=0.08];}}
\node[below right] at (10,5) {$y$};
\end{tikzpicture}
\caption{An illustration of the sets $\um(y)$ (represented by white circles) and
\protect\rotatebox[origin=c]{180}{$\angle$}$(y)$ (represented by filled black circles),
with $y=(x,n) \in \Z^2$.}
\label{f:cones}
\end{figure}

Define $\bar v = \tfrac13v_\star$. Let $\um(x,n)$ be the cone in the first quadrant based
at $(x,n)$ with angle $\bar v$, i.e.,
\begin{equation}
\um(x,n) = \um(0,0) + (x,n), \text{ where } \um(0,0) = \{(x,n) \in \mathbb{Z}_+^2; x \geq \bar v n\},
\end{equation}
and $\tres(x,y)$ the cone in the third quadrant based at $(x,n)$ with angle $\bar v$, i.e.,
\begin{equation}
\tres(x,n) = \tres(0,0) + (x,n), \text{ where } \tres(0,0) = \{(x,n) \in \mathbb{Z}_-^2\colon\,x < \bar v n\}.
\end{equation}
(See Figure~\ref{f:cones}.) Note that $(0,0)$ belongs to $\um(0,0)$ but not to $\tres(0,0)$.

Define the following sets of trajectories in $W$:
\begin{equation}
\begin{aligned}
W_{x,n}^\um &= \text{ trajectories that intersect $\um(x,n)$ but not $\tres(x,n)$},\\
W_{x,n}^\tres &= \text{ trajectories that intersect $\tres(x,n)$ but not $\um(x,n)$},\\
W_{x,n}^\treze &= \text{ trajectories that intersect both $\um(x,n)$ and $\tres(x,n)$}.
\end{aligned}
\end{equation}
Note that $W^\um$, $W^\tres$ and $W^\treze$ form a partition of $W$. As above, we write $Y_n$
to denote $Y^{0}_n$. For $y\in \Z^2$, define the sigma-algebras
\begin{equation}
\label{e:sigmaalgebrastraj}
\mathcal{G}^{I}_{y} = \sigma \left( \omega(A) \colon\, A \subset W^{I}_{y}, A \in \cW  \right),
I = \um,\tres,\treze,
\end{equation}
and note that these are jointly independent under $\P$. Also define the sigma-algebras
\begin{equation}
\label{e:sigmaalgebraunif}
\begin{aligned}
\mathcal{U}^{\um}_{y} & = \sigma \left( U_z \colon\, z \in \um(y) \right),\\
\mathcal{U}^{\tres}_{y} & = \sigma \left( U_{z} \colon\, z \in \tres(y) \right),
\end{aligned}
\end{equation}
and set
\begin{equation}
\label{e:sigmaalgebraFxt}
\mathcal{F}_{y} = \mathcal{G}^{\tres}_{y} \vee \mathcal{G}^{\treze}_{y} \vee \mathcal{U}^{\tres}_{y}.
\end{equation}

Next, define the \emph{record times}
\begin{equation}
\label{e:records}
R_k = \inf \{n\in{\mathbb{Z}_+}\colon\, X_n \ge (1-\bar{v})k + \bar{v} n \}, \qquad k \in \N,
\end{equation}
i.e., the time when the walk first enters the cone
\begin{equation}
\label{e:defcones}
\um_k = \um((1-\bar{v})k,0).
\end{equation}
Note that, for any $k \in \N$, $y \in \um_k$ if and only if $y+(1,1) \in \um_{k+1}$. Thus ,
$R_{k+1} \ge R_k+1$, and $X_{R_k+1}-X_{R_k}=1$ if and only if $R_{k+1} = R_k+1$.

Define a filtration $\mathcal{F}
= (\mathcal{F}_k)_{k \in \N}$ by setting $\mathcal{F}_{\infty} = \sigma\left( \omega(A) \colon\, A
\in \cW \right) \vee \sigma (U_{y} \colon y \in \Z^2)$ and
\begin{equation}
\label{e:filtration}
\mathcal{F}_k = \Big\{B \in \mathcal{F}_{\infty} \colon\, \,
\forall \, y \in \Z^2, \, \exists \, B_{y} \in \mathcal{F}_{y}
\text{ with } B \cap \{Y_{R_k} = y\} = B_{y} \cap \{Y_{R_k} = y\} \Big\},
\end{equation}
i.e., the sigma-algebra generated by $Y_{R_k}$, all $U_z$ with $z \in \tres(Y_{R_k})$
and all $\omega(A)$ such that $A \subset W^\tres_{Y_{R_k}} \cup W^\treze_{Y_{R_k}}$.
In particular, $(Y_i)_{0 \le i \le R_k} \in \mathcal{F}_k$.

Finally, define the event
\begin{equation}
\label{e:Axt}
A^{y} = \big\{Y^{y}_i \in \um(y) \,\,\forall\,i \in {\mathbb{Z}_+} \big\},
\end{equation}
in which the walker remains inside the cone $\um(y)$,
the probability measure
\begin{equation}
\mathbb{P}^{\um} (\cdot) = \mathbb{P} \left(~\cdot~ {\big |}
~\omega\big(W^{\treze}_{0}\big)=0,\, A^{0}\right),
\label{e:pmarrom}
\end{equation}
the \emph{regeneration record index}
\begin{equation}
\label{e:regrec}
\mathcal{I} = \inf\Big\{ k \in \N \colon\, \omega\big(W^{\treze}_{Y_{R_k}}\big)= 0,\,
A^{Y_{R_k}} \text{ occurs } \Big\}
\end{equation}
and the \emph{regeneration time}
\begin{equation}
\label{e:regtime}
\tau = R_{\mathcal{I}}.
\end{equation}


\subsection{Regeneration theorems}
\label{ss:regeneration}

The following two theorems are our key results for the regeneration times.

\begin{theorem}
\label{t:regeneration}
Almost surely on the event $\{\tau < \infty\}$, the process $(Y_{\tau+i} - Y_\tau)_{i \in\Z_+}$
under either the law $\mathbb{P} (~\cdot \mid \tau, (Y_i)_{0\le i \le \tau})$ or $\mathbb{P}^{\um}
(~\cdot \mid \tau, (Y_i)_{0 \le i \le \tau})$ has the same distribution as that of $(Y_i)_{i \in\Z_+}$
under $\mathbb{P}^\um(\cdot)$.
\end{theorem}

\newconstant{c:tailreg}

\begin{theorem}
\label{t:tailregeneration}
There exists a constant $\useconstant{c:tailreg} > 0$ such that
\begin{equation}
\label{e:tailregeneration}
\E \left[e^{\useconstant{c:tailreg} \log^\gamma \tau} \right] < \infty
\end{equation}
and the same holds under $\mathbb{P}^\um$.
\end{theorem}


\subsubsection{Proof of Theorem~\ref{t:regeneration}}

\begin{proof}
First we observe that, for all $k \in \N$ and all bounded measurable functions $f$,
\begin{equation}
\label{e:regmakessense}
\begin{aligned}
& \mathbb{E}\Big[ f\big( (Y_{R_k+i} - Y_{R_k})_{i \in {\mathbb{Z}_+}} \big),
A^{Y_{R_k}}  \mid \mathcal{F}_k \Big] \\
& \qquad \qquad = \mathbb{E}^{\um} \left[f((Y_i)_{i\geq 0}) \right] \mathbb{P}\left( A^{0}
\mid  \omega(W^{\treze}_{0})=0 \right) \;\; \text{a.s.\ on } \;
\omega(W^{\treze}_{Y_{R_k}})=0.
\end{aligned}
\end{equation}
Indeed, we have
\begin{equation}
\label{e:regmakessense2}
\begin{aligned}
& \mathbb{E}\left[ f((Y^{y}_{i})_{i\geq 0}), A^{y}, B_{y}, \omega(W^{\treze}_{y})=0, Y_{R_k} = y \right] \\
& \qquad  = \mathbb{E}^{\um} \left[f((Y_i)_{i \geq 0}) \right] \mathbb{P}\left( A^{0} \mid
\omega(W^{\treze}_{0})=0 \right) \mathbb{P}\left( B_{y},
\omega(W^{\treze}_{Y_{R_k}})=0, Y_{R_k} = y \right)
\end{aligned}
\end{equation}
for all $B_{y} \in \mathcal{F}_{y}$ because
\begin{enumerate}[(1)]
\label{verifregmakessense2}
\item
$\omega(W^{\treze}_{y}), \{Y_{R_k} = y\} \in \mathcal{F}_{y}$.
\item
On $\omega(W^{\treze}_{y})=0$, $f((Y^{y}_i)_{i \geq 0}) \mathbf{1}_{A^{y}}
\in \mathcal{G}^{\um}_{y} \vee \mathcal{U}^{\um}_{y}$.
\item
The joint distribution of $f((Y^{y}_i)_{i \geq 0})$, $A^{y}$ and $\omega(W^{\treze}_{y})$ under $\mathbb{P}$
does not depend on $y$.
\end{enumerate}
By summing \eqref{e:regmakessense2} over $y \in \Z^2$, we get \eqref{e:regmakessense}.

Next, let $\mathcal{F}_\tau$ be the sigma-algebra of the events before time $\tau$, i.e., the set of
all events $B \in \mathcal{F}_{\infty}$ such that, for each $k \in \N$, there exists a $B_k \in \mathcal{F}_k$
such that $B \cap \{\mathcal{I} = k\} = B_k \cap \{\mathcal{I} = k\}$. Note that $\tau$ and $(Y_i)_{0 \le i \le \tau}$
are measurable with respect to $\mathcal{F}_\tau$. Let $\Gamma_k = \{\omega(W^{\treze}_{Y_{R_k}})=0\}
\cap A^{Y_{R_k}}$, and note that for each $0 \le k \le n \in \N$ there exists a $D_{k,n} \in \mathcal{F}_n$
such that $\Gamma_k \cap \Gamma_n = D_{k,n} \cap \Gamma_n$. In particular, there exists a
$C_n \in \mathcal{F}_n$ such that
\begin{equation}
\{\mathcal{I} = n\} = \bigcap_{k=1}^{n-1} \Gamma_k^c \cap \Gamma_n = C_n \cap \Gamma_n.
\end{equation}
Thus, for $B \in \mathcal{F}_\tau$ and $f$ bounded measurable, we may write
\begin{align}
\label{eq:verif_reg1}
& \mathbb{E} \Big[ f\big((Y_{\tau+i}-Y_\tau)_{i \in {\mathbb{Z}_+}}\big), B, \tau < \infty \Big]
\nonumber\\
& = \sum_{n\in\N} \mathbb{E} \Big[ f\big((Y_{R_n+i}-Y_{R_n})_{i \in{\mathbb{Z}_+}}\big),
B_n, C_n, \Gamma_n \Big]
\nonumber \\
& = \sum_{n\in\N} \mathbb{E} \Big[ B_n, C_n, \omega(W^{\treze}_{Y_{R_n}})=0,
\mathbb{E} \big[ f\big((Y_{R_n+i}-Y_{R_n})_{i \in{\mathbb{Z}_+}}\big),
A^{Y_{R_n}} \mid \mathcal{F}_n \big] \Big].
\end{align}
By \eqref{e:regmakessense}, the right-hand side equals
\begin{align}\label{eq:verif_reg2}
\mathbb{E}^\um \left[f(Y)\right] \sum_{n\in\N}
\mathbb{P} \left( B_n, C_n, \omega(W^{\treze}_{Y_{R_n}})=0 \right)
\mathbb{P}\left(A^{0} \;\middle|\; \omega(W^{\treze}_{0})=0\right),
\end{align}
which, again by \eqref{e:regmakessense}, equals
\begin{align}
\label{eq:verif_reg3}
&\mathbb{E}^\um \left[f(Y)\right] \sum_{n\in\N}
\mathbb{E} \Big[ B_n, C_n, \omega(W^{\treze}_{Y_{R_n}})=0,
\mathbb{P} \left( A^{Y_{R_n}} \mid \mathcal{F}_n \right) \Big]
\nonumber\\
& \qquad \qquad = \mathbb{E}^\um \left[f(Y)\right] \sum_{n\in\N}
\mathbb{P} \left(B_n, \mathcal{I} = n \right)
\nonumber\\
& \qquad \qquad = \mathbb{E}^\um \left[f(Y)\right]
\mathbb{P} \left( B, \tau < \infty \right),
\end{align}
which proves the statement under $\mathbb{P}(\cdot)$. To extend the result to $\mathbb{P}^{\um}
= \mathbb{P}(\cdot \mid \Gamma_0)$, note that $\Gamma_0 \in \mathcal{F}_{\tau}$ because
$\Gamma_0 \cap \Gamma_n = D_{0,n} \cap \Gamma_n$ with $D_{0,n} \in \mathcal{F}_n$, and
so we may apply \eqref{eq:verif_reg3} to $B \cap \Gamma_0$.
\end{proof}


\subsubsection{Proof of Theorem~\ref{t:tailregeneration}}

In what follows the constants may depend on $v_\circ,v_\bullet$, $v_\star$ and $\rho$.
We begin with a few preliminary lemmas.

Define the \emph{influence field at a point $y \in \Z^2$} as
\begin{equation}
\label{e:hxt}
h(y) = \inf \Big\{ l \in {\mathbb{Z}_+}\colon\, \omega(W^{\treze}_y \cap W^{\treze}_{y + (l,l)}) = 0 \Big\}.
\end{equation}

\newconstant{c:h_xt1}
\newconstant{c:h_xt2}

\begin{lemma}
\label{l:hxt_exp}
There exist constants $\useconstant{c:h_xt1}, \useconstant{c:h_xt2} > 0$ (depending on
$v_\star, \rho$ only) such that, for all $y \in \mathbb{Z}^2$,
\begin{equation}
\label{e:h_xt_exp}
\mathbb{P}[h(y) > l] \leq \useconstant{c:h_xt1} e^{-\useconstant{c:h_xt2} l}, \qquad l\in{\mathbb{Z}_+}.
\end{equation}
\end{lemma}

\begin{proof}
By translation invariance, it is enough to consider the case $y = 0$. By the definition of $h(0)$,
we know that
\begin{equation}
\{h(0) > l\} \subseteq \big\{\exists\, y \in \tres(0,0), y' \in \um(l,l)\colon\,
\omega(W_{y \leftrightarrow y'}) > 0 \big\},
\end{equation}
where $W_{(x,n) \leftrightarrow (x',n')} = \{w \in W\colon\, w(n) = x, w(n') = x'\}$ (recall
\eqref{e:T_rho}). It follows that
\begin{equation}
\label{e:bound_Ph0}
\begin{split}
&\mathbb{P}(h(0) > l) \leq \sum_{y \in \tres (0,0)} \sum_{y' \in \um (l,l)}
\mathbb{P} (\omega(W_{y \leftrightarrow y'}) > 0)\\
& = \sum_{y \in \tres (0,0)} \sum_{y' \in \um (l,l)}
\left(1 - e^{-\rho \mu(W_{y \leftrightarrow y'})} \right)
\leq \sum_{y \in \tres (0,0)} \sum_{y' \in \um (l,l)} \rho \mu(W_{y \leftrightarrow y'}).
\end{split}
\end{equation}

Recall from Section~\ref{ss:DRE} that $P_x$ stands for the law on $W_x$ under which the family
$(Z_n)_{n \in \mathbb{Z}}$ given by $Z_n(w) = w(n)$ is distributed as a two-sided simple random
walk starting at $x$. We write $y = (x,n)$ and $y' = (x',n')$, use translation invariance of $\mu$,
and use Azuma's inequality, to get
\begin{equation}
\label{e:right_by_l}
\mu(W_{y \leftrightarrow y'}) = P_0 (Z_{n' - n} = x'- x) \leq
P_0 (Z_{n'  - n} \geq x' - x) \leq \exp \Big\{ -\frac{(x' - x)^2}{2(n' - n)} \Big\}.
\end{equation}

Combining (\ref{e:bound_Ph0}--\ref{e:right_by_l}) and noting that $n' - n \leq (x' - x)/\bar v$,
we get
\begin{equation}
\label{e:Prho_expxx}
\mathbb{P}(h(0) > l) \leq \rho \sum_{(x,n) \in \tres (0,0)} \sum_{(x',n') \in \um (l,l)}
\exp\left\{- \tfrac12 \bar v (x' - x)\right\}.
\end{equation}
For fixed $x = -k$, there are at most $k/{\bar v}$ space-time points $(x,n) \in \tres (0,0)$. Analogously,
for  fixed $x' = k' + l$, there are at most $(k'+1)/{\bar v}$ space-time points $(x', t') \in \um (l,l)$. Therefore,
using \eqref{e:Prho_expxx} we obtain
\begin{equation}
\mathbb{P}(h(0) > l) \leq \frac{\rho}{{\bar v}^2} \sum_{k,k' \in \mathbb{Z}_+} \, k (k'+1)\, e^{-{\bar v}(k + k' + l)/2}
\leq \frac{\rho}{{\bar v}^2}\, e^{-{\bar v}l/2} \Big( \sum_{k \in {\mathbb{Z}_+}} (k+1) e^{-{\bar v}k/2}\Big)^2.
\end{equation}
By choosing the constants $\useconstant{c:h_xt1}$ and $\useconstant{c:h_xt2}$ properly, we get
the claim.
\end{proof}

Choose
\begin{equation}
\label{e:eps}
\delta = \frac{1}{4 \log\big( \tfrac{1}{p_\circ \wedge p_{\bullet}} \big)}, \qquad
\epsilon = \frac{1}{4}(\useconstant{c:h_xt2} \delta \wedge 1),
\end{equation}
and put, for $T>1$,
\begin{equation}
T' = \lfloor T^\epsilon \rfloor,\qquad
T'' = \lfloor \delta \log(T) \rfloor.
\end{equation}
Define the \emph{local influence field at $(x,n)$} as
\begin{equation}
\label{e:hxt_local}
h^T(x,n) = \inf \big\{ l \in{\mathbb{Z}_+}\colon\, 
\omega( W^\um_{x-\lfloor (1-\bar{v}) T' \rfloor,n} \cap W^{\treze}_{x,n}
\cap W^\treze_{x + l, n + l}) = 0 \big\}.
\end{equation}

\begin{lemma}
\label{l:locinfl}
For all $T > 1$,
\begin{equation}
\label{e:locinfl}
\mathbb{P} \left( h^T(y) > l \;\middle|\; \mathcal{F}_{y-(\lfloor (1-\bar{v}) T' \rfloor,0)} \right)
\le \useconstant{c:h_xt1} e^{-\useconstant{c:h_xt2} l}, \quad \mathbb{P}\text{-a.s.}
\qquad \forall \,y \in \Z^2,\,l \in \Z_+,
\end{equation}
where $\useconstant{c:h_xt1}, \useconstant{c:h_xt2}$ are the same constants as in
Lemma~{\rm \ref{l:hxt_exp}}.
\end{lemma}
\begin{proof}
The result follows from Lemma~\ref{l:hxt_exp} by noting that $h^T(y)$ is independent
of $\mathcal{F}_{y - (\lfloor (1-\bar{v}) T' \rfloor,0)}$ and smaller than $h(y)$.
\end{proof}

We say that $R_k$ is a \emph{good record time (g.r.t.)} when
\begin{align}
\label{e:good_record1}
& h^T(Y_{R_k}) \leq T'',\\
\label{e:good_record2}
& U_{Y_{R_k} + (l,l)} \leq p_\circ \wedge p_{\bullet}, \quad \forall\,l = 0, \dots, T''-1,\\
\label{e:good_record3}
& \omega(W^\um_{Y_{R_k}} \cap W^{\treze}_{Y_{R_k}+(T'', T'')}) = 0,\\
\label{e:good_record4}
& \{ Y_{R_{k+ T''}+1}, \ldots, Y_{R_{k+ T'}} \} \subset \um(Y_{R_{k+T''}}).
\end{align}
Note that, when \eqref{e:good_record2} occurs, $Y_{R_k}+(T'',T'') = Y_{R_{k+T''}}$
(see Fig.~\ref{f:Rk_good}).

\begin{figure}[htbp]
\centering
\begin{tikzpicture}[scale=.8]
\draw[dashed] (0,0) -- (3,9); 
\draw[dashed] (5,0) -- (8,9); 
\draw[dashed] (10,0) -- (13,9); 
\draw[<->] (0,0) -- (5,0); \node (Rk) at (2.5,0) [below]{$(1-\bar{v}) T'$}; 
\draw [fill] (0.1,0.3) circle [radius=0.05]; \node (Rk) at (0.1,0.3) [above left]{$Y_{R_{k-T'}}$};
\draw[thick] plot [smooth,tension=.5] coordinates{(0.1, 0.3) (1.2, 0.8) (2.0, 1.4) (2.6, 1.8)
(3.4, 2.0) (3.8, 2.4) (4.4, 2.8) (5.2, 3.0) (5.7, 3.2) (6.4, 4.2)};
\draw [fill] (6.4,4.2) circle [radius=0.05]; \node (Rk) at (6.4,4.2) [above left]{$Y_{R_{k}}$};
\draw[dotted] (6.4,4.2) -- (6.4,0);
\draw[-<] plot [smooth,tension=1] coordinates{(6.4,6.2) (6.9,5.6) (7.1,4.9)};
\node at (6.4,6.2) [above left]{\eqref{e:good_record2}};
\node (Rk) at (6.4,4.2) [below right]{$h^T(Y_{R_{k}}) \leq T''$, \eqref{e:good_record1}};
\draw[thick, dashed] (6.4,4.2) -- (8,5.4); 
\draw [fill] (8,5.4) circle [radius=0.05]; \node (Rk) at (8,5.4) [below right]{$Y_{R_{k+T''}}$};
\draw[dotted] (8,5.4) -- (8,4.4); \draw[dotted] (8,3.2) -- (8,0);
\draw[<->] (6.4,0) -- (8,0); \node (Rk) at (7.2,0) [below]{$T''$}; 
\draw (9.3,9) -- (8,5.4) -- (11.8,5.4);
\draw[thick] plot [smooth,tension=.5] coordinates{(8.0, 5.4) (8.4, 5.6) (9.1, 5.9) (9.7, 6.2)
(10.4, 6.4) (11.0, 6.8) (11.3, 7.2) (12.0, 7.6) (12.4, 7.9) (12.8, 8.4)};
\node at (11,6.8) [above left]{\eqref{e:good_record4}};
\draw [fill] (12.8,8.4) circle [radius=0.05]; \node (Rk) at (12.8,8.4) [below right]{$Y_{R_{k+T'}}$};
\end{tikzpicture}
\caption{\small Illustration of a good record time $R_k$. Note the validity of the conditions
\eqref{e:good_record1}, \eqref{e:good_record2} and \eqref{e:good_record4}.}
\label{f:Rk_good}
\end{figure}
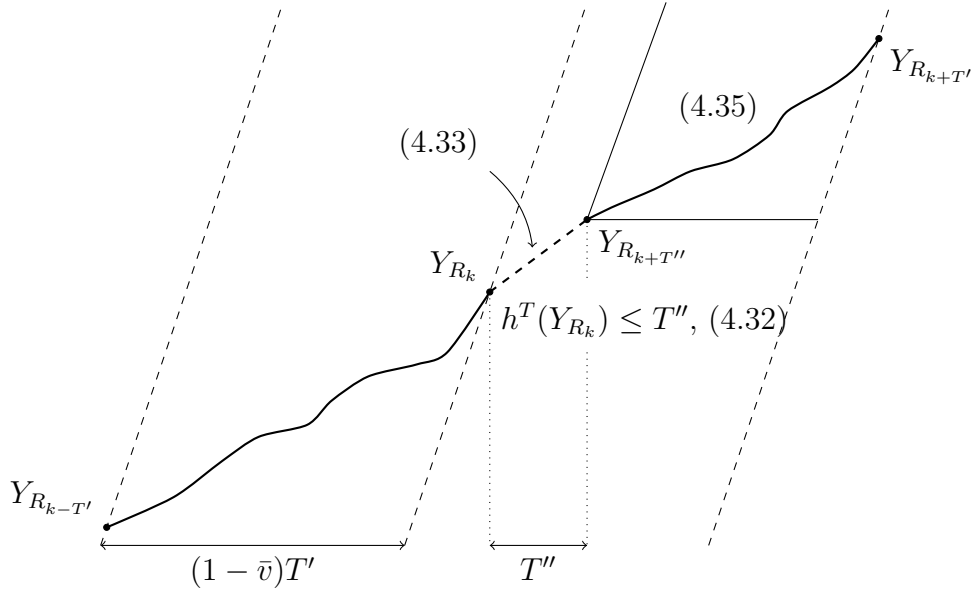

The idea is that, when $R_k$ is a good record time, $R_{k+T''}$ is likely to be an upper bound
for the regeneration time. In Proposition~\ref{p:manygrts} below we will show that when many
records are made, with high probability good record times occur. First, we need an
additional lemma.

For $y \in \Z^2$, denote by
\begin{equation}\label{e:defkappay}
\kappa(y) := \max \{ k \in \N \colon\, y \in \um_k\}
\end{equation}
the index of the last cone containing $y$. Note that $\kappa(Y_{R_k}) = k$.
Then define, for $t \in \N$, the space-time parallelogram
\begin{equation}
\label{defparallel}
\mathcal{P}_t(y) = \left( \um(y) \setminus \um_{\kappa(y)+t} \right) 
\cap \left( y + \{(x,n) \in \Z^2 \colon\, n \le t/\bar{v} \} \right)
\end{equation}
and its right boundary
\begin{equation}
\label{defrightbound}
\partial^+\mathcal{P}_t(y) = \{z \in \Z^2  \setminus \mathcal{P}_t(y) \colon\, z - (1,0)
\in \mathcal{P}_t(y) \}.
\end{equation}
We say that ``$Y^y \text{ exits } \mathcal{P}_t(y) \text{ through the right}$'' when the
first time $i$ at which $Y^y_i \notin \mathcal{P}_t(y)$ satisfies $Y^y_i \in \partial^+
\mathcal{P}_t(y)$. 
Note that, when $y=Y_{R_k}$, this implies $Y^y_i = Y_{R_{k+t}}$.

\newconstant{c:height}

\begin{lemma}
\label{l:no_top}
There exists a constant $\useconstant{c:height} > 0$ such that, for all $t \in \N$ large enough,
\begin{equation}
\label{e:no_top}
\mathbb{P} \left( Y^y \text{ exits } \mathcal{P}_t(y)
\text{ through the right} \;\middle| \; \mathcal{F}_y \right)
\ge \useconstant{c:height} \quad \mathbb{P}\text{-a.s.} \qquad \forall\,y \in \Z^2.
\end{equation}
\end{lemma}

\begin{proof}
If $v_\circ \ge v_\bullet$, then the claim follows from simple random walk estimates, since
$0 < \bar v < v_\bullet$. Therefore we may assume that $v_\circ < v_\bullet$.

First note that, for fixed $y$ and $t$ large enough (e.g.\ $t > 3$),
\begin{align}
\label{e:tright0}
\mathbb{P} \left( Y^y \text{ exits } \mathcal{P}_t(y)
\text{ through the right} \;\middle| \; \mathcal{F}_y \right)
\ge \; & \mathbb{P} \left( Y^y_n -y \notin \mathcal{H}_{v_\star,1} \,\forall\, n \in \Z_+ \;\middle| \; \mathcal{F}_y \right)
\end{align}
where $\mathcal{H}_{v,L}$ is as in \eqref{e:cal_H}.
Reasoning as for \eqref{e:regmakessense2}, we see that the latter equals
\begin{align}
\label{e:tright1}
\mathbb{P} \left( Y_n \notin \mathcal{H}_{v_\star,1} \,\forall\, n \in \Z_+ \;\middle| \; \omega(W^{\treze}_0) =0 \right)
\;\; \mathbb{P}\text{-a.s.\ on the event } \{\omega(W^{\treze}_{y}) = 0\}.
\end{align}
By monotonicity, if $\omega(W^{\treze}_{y}) > 0$, then $Y^y$ can only be further to the right.
Hence
\begin{align}
\label{e:tright2}
& \mathbb{P} \left( Y^y \text{ exits } \mathcal{P}_t(y)
\text{ through the right} \;\middle| \; \mathcal{F}_y \right) \nonumber\\
& \ge \mathbb{P} \left( Y_n \notin \mathcal{H}_{v_\star,1} \,\forall\, n \in \Z_+ \;\middle| \; \omega(W^{\treze}_0) =0 \right) 
\;\; \mathbb{P}\text{-a.s.,}
\end{align}
so we only need to show that this last probability is strictly positive. 
To that end, fix $L > 1$ large enough such that
\begin{equation}
\label{e:tright3}
\mathbb{P} \left(\exists\,n \in {\mathbb{Z}_+}\colon\,Y_n \in
\mathcal{H}_{v_\star, \lfloor L(1-v_\star) \rfloor}\right)
\leq \tfrac12 \; \mathbb{P} \left( \omega(W^{\treze}_0) =0 \right),
\end{equation}
which is possible by \eqref{e:LD}. If $t$ is large enough (e.g.\ $t > 2$), then
\begin{align}
\label{e:tright4}
& \mathbb{P} \left( Y_n \notin \mathcal{H}_{v_\star,1} \,\forall\, n \in \Z_+,
\, \omega(W^{\treze}_0) =0 \right) \nonumber\\
& \ge \mathbb{P} \left( U_{(i,i)} \le p_\circ \; \forall \; i = 0, \ldots, L-1, \; Y^{(L,L)}_n
\notin \mathcal{H}_{v_\star, 0} \,\forall\, n \in \Z_+, \; \omega(W^{\treze}_0) =0 \right) \nonumber\\
& \ge p_\circ^L \left\{ \mathbb{P} \left( \omega(W^{\treze}_0) =0 \right)
- \mathbb{P} \left(\exists\,n \in {\mathbb{Z}_+}\colon\,Y_n \in
\mathcal{H}_{v_\star, \lfloor L(1-v_\star) \rfloor}\right) \right\} \nonumber\\
& \ge \tfrac12 p_\circ^L \mathbb{P} \left( \omega(W^{\treze}_0) =0 \right),
\end{align}
which concludes the proof.
\end{proof}

The following proposition is the main step in the proof of Theorem~\ref{t:tailregeneration}.
\newconstant{c:manygrts}

\begin{proposition}
\label{p:manygrts}
There exists a constant $\useconstant{c:manygrts} >0$ such that, for all $T>1$ large enough,
\begin{equation}
\mathbb{P}\left[\text{$R_k$ is not a g.r.t.\ for all $1\le k \le T$ }\right]
\leq e^{ -\useconstant{c:manygrts} T^{1/2}}.
\end{equation}
\end{proposition}

\begin{proof}
First we claim that there exists a $c>0$ such that, for any $k \ge T'$,
\begin{equation}
\label{e:saw_pemba}
\mathbb{P} \left[ R_k \text{ is a g.r.t.} \big| \mathcal{F}_{k-T'} \right]
\geq c T^{\delta \log(p_\circ \wedge p_{\bullet})}  \text{ a.s.}
\end{equation}
To prove \eqref{e:saw_pemba}, we will find $c>0$ such that
\begin{align}
\label{e:good_cond1}
& \mathbb{P} \big[ \text{\eqref{e:good_record1}} \; \big| \; \mathcal{F}_{k-T'} \big] \geq c
 & \text{ a.s.,}\\
\label{e:good_cond2}
& \mathbb{P} \big[ \text{\eqref{e:good_record2}} \; \big| \; \mathcal{F}_k \big]
\geq T^{\delta \log(p_\circ \wedge p_{\bullet})}  & \text{ a.s.,}\\
\label{e:good_cond3}
& \mathbb{P} \big[ \text{\eqref{e:good_record3}} \; \big| \; \text{\eqref{e:good_record2}},
\mathcal{F}_k \big] \geq c  & \text{ a.s.,}\\
\label{e:good_cond4}
& \mathbb{P} \big[ \text{\eqref{e:good_record4}} \; \big| \; \mathcal{F}_{k+T''} \big] \geq c
 & \text{ a.s.\ }
\end{align}

\eqref{e:good_cond1}: For $B \in \mathcal{F}_{k-T'}$, write
\begin{align}
\label{e:manygrts1}
& \mathbb{P} \left(h^T(Y_{R_k}) > T'', B \right)
= \sum_{y_1, y_2 \in \Z^2} \mathbb{P} \left(h^T(y_2) > T'', Y_{R_k}
= y_2, Y_{R_{k-T'}}=y_1, B_{y_1} \right) \nonumber\\
& \quad \le \sum_{y_1 \in \Z^2} \sum_{y_2 \in \partial^+\mathcal{P}_{T'}(y_1)}
\mathbb{P} \left(h^T(y_2) > T'', Y_{R_{k-T'}}=y_1, B_{y_1} \right) \nonumber\\
& \quad \;\;\; + \sum_{y_1 \in \Z^2} \mathbb{P} \left( Y^{y_1} \text{ does not exit }
\mathcal{P}_{T'}(y_1) \text{ through the right}, Y_{R_{k-T'}}=y_1, B_{y_1} \right) \nonumber\\
& \quad \le \left\{ \useconstant{c:h_xt1} T' e^{-\useconstant{c:h_xt2} T''}
+ 1-\useconstant{c:height} \right\}  \mathbb{P} \left( B \right)
\le \left\{  \useconstant{c:h_xt1} e^{\useconstant{c:h_xt2}} T^{-\frac34 \delta
\useconstant{c:h_xt2}} + 1-\useconstant{c:height} \right\}  \mathbb{P} \left( B \right),
\end{align}
where for the second line we use Lemmas~\ref{l:locinfl}--\ref{l:no_top} and
$|\partial^+ \mathcal{P}_t(y)| \le t/\bar{v}$, while for the third we use the definition of $\epsilon$.
Thus, for $T$ large enough, \eqref{e:good_cond1} is satisfied with $c = \useconstant{c:height}/2$.

\eqref{e:good_cond2}: This follows from the fact that $(U_{Y_{R_k}+(l,l)})_{l \in \N_0}$ is
independent of $\mathcal{F}_k$.

\eqref{e:good_cond3}: We may ignore the conditioning on \eqref{e:good_record2} since
this event is independent of the others. For $B \in \mathcal{F}_k$, write
\begin{align}
\label{e:manygrts2}
& \mathbb{P} \left(\omega(W^{\um}_{Y_{R_k}}
\cap W^{\treze}_{Y_{R_k}+(T'', T'')} ) =0, B\right) \nonumber\\
& \quad = \sum_{y \in \Z^2} \mathbb{P} \left(\omega(W^{\um}_{y}
\cap W^{\treze}_{y+(T'', T'')} ) =0, Y_{R_k} = y, B_y\right) \nonumber\\
&\quad = \sum_{y \in \Z^2} \mathbb{P} \left(\omega(W^{\um}_{y}
\cap W^{\treze}_{y+(T'', T'')} ) =0 \right) \mathbb{P}\left( Y_{R_k} = y, B_y\right) \nonumber\\
& \quad \ge \mathbb{P}\left(\omega(W^{\treze}_0) = 0 \right) \mathbb{P} \left( B \right),
\end{align}
where the second equality uses the independence of $\mathcal{G}^{\um}_y$ and $\mathcal{F}_y$,
and the last step uses the monotonicity and translation invariance of $\omega$.

\eqref{e:good_cond4}: For $B \in \mathcal{F}_{k+T''}$, write
\begin{align}
\label{e:manygrts3}
& \mathbb{P} \left( \{ Y_{R_{k+T''}+1}, \ldots, Y_{R_{k+T'}}\}
\subset \um(Y_{R_{k+T''}}), B \right) \nonumber\\
& \quad \ge \mathbb{P} \left( Y^{Y_{R_{k+T''}}} \text{ exits }
\mathcal{P}_{T'}(Y_{R_{k+T''}}) \text{ through the right}, B \right) \nonumber\\
& \quad = \sum_{y \in \Z^2} \mathbb{P} \left( Y^{y} \text{ exits }
\mathcal{P}_{T'}(y) \text{ through the right}, Y_{R_{k+T''}}=y, B_y \right) \nonumber\\
& \quad \ge \useconstant{c:height} \mathbb{P} \left( B \right)
\end{align}
by Lemma~\ref{l:no_top}.

Thus, \eqref{e:saw_pemba} is verified. Since $\{R_k \text{ is a g.r.t.}\} \in \mathcal{F}_{k+T'}$,
we obtain, for $T$ large enough,
\begin{align}
\label{e:manygrts4}
\mathbb{P} \left( R_k \text{ is not a g.r.t.\ for any } k \le T \right)
& \le \mathbb{P} \left( R_{(2k+1)T'} \text{ is not a g.r.t.\ for any } k \le T/3T' \right) \nonumber\\
& \le \exp \left\{ -\frac{c}{4} \frac{T^{1+{\delta \log(p_\circ \wedge p_\bullet)}}}{T'} \right\} \nonumber\\
& \le \exp \left\{-\frac{c}{4}T^{\frac12} \right\}
\end{align}
by our choice of $\epsilon$ and $\delta$.
\end{proof}

We are now ready to finish the proof of Theorem~\ref{t:tailregeneration}.

\begin{proof}[Proof of Theorem~{\rm \ref{t:tailregeneration}}]
Since $\mathbb{P}^{\um}(\cdot) = \mathbb{P}(\cdot | \Gamma_0)$ with
$\mathbb{P}(\Gamma_0)>0$, it is enough to prove the statement under
$\mathbb{P}$. To that end, let
\begin{equation}
\label{prtailreg1}
\begin{array}{lcl}
E_1 & = & \{\exists \; y \in [-T,T] \times [0,T] \cap \Z^2 \colon\, h(y) \ge T'/2\},\\
E_2 & = & \{\exists \; y \in [-T,T] \times [0,T] \cap \Z^2 \colon\, Y^y \text{ touches } y
+ \mathcal{H}_{v_\star, \lfloor T'/2 \rfloor}\}.
\end{array}
\end{equation}
Then, by Lemma~\ref{l:hxt_exp}, \eqref{e:LD} and a union bound, there exists
a $c>0$ such that
\begin{equation}\label{prtailreg2}
\mathbb{P} \left( E_1 \cup E_2 \right) \le c^{-1} e^{-c  \log^\gamma T} \qquad \forall\,T>1.
\end{equation}
Next we argue that, for $T$ large enough, if $R_k$ is a g.r.t.\ with $k \le {\bar v}T$ and
both $E_1$ and $E_2$ do not occur, then $\tau \le R_{k+T''} \le T$. Indeed, if $T'' < T'
< {\bar v}T/2$, then on $E_2^c$ we have $R_{\lfloor {\bar v}T \rfloor +T'} \le T$, since
otherwise $Y$ touches $\mathcal{H}_{v_\star, \lfloor T'/2 \rfloor}$. 
Thus, all we need to verify is that
\begin{equation}
\label{prtailreg3}
\omega(W^\treze_{Y_{R_{k+T''}}})=0
\end{equation}
and that
\begin{equation}
\label{prtailreg4}
A^{Y_{R_{k+T''}}} \text{ occurs }
\end{equation}
under the conditions stated.

To verify \eqref{prtailreg4}, note that $Y_{R_{k+T'}} \in [-T,T] \times[0,T] \cap \Z^2$
on $E_1^c \cap E_2^c$. 
Therefore
\begin{equation}\label{prtailreg4.5}
Y_{R_{k+T'}+l} \in \um(Y_{R_{k+T''}}) \; \forall \; l \in \Z_+
\end{equation}
on $E_2^c$, and \eqref{prtailreg4} follows from \eqref{prtailreg4.5} and \eqref{e:good_record4}.

To verify \eqref{prtailreg3}, first note that, by \eqref{e:good_record1} and \eqref{e:good_record3},
we only need to check that
\begin{equation}
\label{prtailreg5}
\omega( W^{\treze}_{Y_{R_{k+T''}}} \cap  W^\treze_{Y_{R_k}} \cap W^{\treze}_{Y_{R_k} - (\lfloor (1-\bar{v}) T' \rfloor,0)} )
\le \omega(W^\treze_{Y_{R_k}} \cap W^{\treze}_{Y_{R_k} - ( \lfloor (1-\bar{v}) T' \rfloor,0)} ) = 0.
\end{equation}
To that end, define
\begin{equation}
\label{prtailreg6}
\mathcal{L} = \tres(Y_{R_k}-(\lfloor (1-\bar{v}) T' \rfloor, 0)) \cap \Z \times [0,\infty).
\end{equation}
Note that $\partial^+\mathcal{L} \subset [-T,T] \times [0,T]$ on $E_1^c \cap E_2^c$. 
Furthermore, since the paths in $W$ take nearest-neighbour steps,
we have
\begin{equation}
\{ \omega(W^\treze_{Y_{R_k}} \cap W^{\treze}_{Y_{R_k} - ( \lfloor (1-\bar{v}) T' \rfloor,0)} ) >0 \}
\subset \{\exists \; y \in \partial^+ \mathcal{L} \colon\, h(y) \ge T'/2\}
\end{equation}
and the latter set is empty on $E_1^c \cap E_2^c$. Thus, \eqref{prtailreg5} holds.

In conclusion, for $T$ large enough we have
\begin{align}
\label{prtailreg8}
\mathbb{P} \left( \tau > T \right)
& \le \mathbb{P}(E_1 \cup E_2) + \mathbb{P} \left( R_k \text{ is not a g.r.t.\ }
\forall \; k \le {\bar v}T  \right) \nonumber\\
& \le c^{-1}e^{-c(\log T)^\gamma} + e^{-\useconstant{c:manygrts} ({\bar v} T)^{1/2}},
\end{align}
which concludes the proof.
\end{proof}


\subsection{Proof of Theorem~\ref{thm:SLLN+CLT+LD}}
\label{ss:limittheorems}

We begin by making the following observation.

\begin{theorem}
\label{t:regtimes}
On an enlarged probability space there exists a sequence $(\tau_n)_{n \in \N}$ of random times
with $\tau_1 = \tau$ such that, under $\P$ and with $S_n=\sum_{i=1}^n\tau_i$,
\begin{equation}\label{iidsequence}
\Big(\tau_{n+1},\left(X_{S_n+s} - X_{S_n}\right)_{0 \le s \le \tau_{n+1}}\Big)_{n \in \N}
\end{equation}
is an i.i.d.\ sequence, independent of $(\tau,(X_s)_{0 \le s \le \tau})$, with each of its terms
distributed as $(\tau,(X_s)_{0 \le s \le \tau})$ under $\P^{\um}$.
\end{theorem}

\begin{proof}
The claim follows from Theorem~\ref{t:regeneration} and the fact that $\tau < \infty$ a.s., exactly
as in Avena, dos Santos and V\"ollering~\cite[Proof of Theorem 3.8]{AvdSaVo}.
\end{proof}

We are now ready to prove Theorem~\ref{thm:SLLN+CLT+LD}.

\begin{proof}[Proof of Theorem~\ref{thm:SLLN+CLT+LD}]
We start with (c).
Let
\be{defspeed}
\iota:= \E^\um[\tau] \quad \text{ and } \quad v := \frac{\E^\um[X_{\tau}]}{\E^\um[\tau]}.
\ee
By Theorems~\ref{t:tailregeneration} and~\ref{t:regtimes} and Lemma~\ref{l:convrate}, for any
$\varepsilon > 0$ we have
\begin{equation}
\label{prLT1}
\P \left( \exists \; i \ge n \colon \left| X_{S_i}-S_i v \right| \vee \left|S_i - i\iota \right| > \varepsilon i \right)
\le c^{-1} e^{ -c \log^\gamma n}
\end{equation}
for some $c>0$.

Next define, for $t \ge 0$, $k_t$ as the random integer such that
\be{SubSequence}
S_{k_t}\leq t < S_{k_t+1}.
\ee
Since $S_n > t$ if and only if $k_t < n$, for any $\varepsilon > 0$ we have
\begin{equation}\label{prLT2}
\P \left( \exists \; t \ge n \colon \left| k_t - t/\iota \right| > t \varepsilon \right)
\le \P \left( \exists \; i \ge \delta' n \colon |S_i - i \iota| > \varepsilon' i \right)
\end{equation}
for some $\delta', \epsilon' > 0$ and large enough $n$.
On the other hand, since $X$ is Lipschitz we have
\begin{equation}\label{prLT3}
|X_t - t v| \le |X_{S_{k_t}} - S_{k_t} v| + (1+|v|) \left\{ |S_{k_t}-k_t \iota| + |k_t \iota - t| \right\},
\end{equation}
and therefore for any $\varepsilon > 0$
\begin{align}\label{prLT4}
\P \left( \exists \; t \ge n \colon |X_t - t v| > \varepsilon t \right)
& \le \P \left( \exists \; t \ge n \colon \left| k_t - t/\iota \right| > t \varepsilon' \right) \nonumber\\
& \;\;\; + \P \left( \exists \; i \ge \delta' n \colon \left| X_{S_i}-S_i v \right|
\vee \left|S_i - i\iota \right| > \varepsilon i \right)
\end{align}
for some $\delta', \varepsilon' > 0$. Combining \eqref{prLT1}--\eqref{prLT4}, we get (c), and
(a) follows by the Borel-Cantelli lemma.

To prove (b), let $\hat{\sigma}^2$ be the variance of $X_{\tau}-\tau v$ under $\P^\um$, which
is finite because of \eqref{e:tailregeneration} and strictly positive because $X_{\tau}-\tau v$ is not
a.s.\ constant. For the process $(Y_k)_{k\in\N}$ defined by $Y_k = X_{S_k}-S_k v$, a functional
central limit theorem with variance $\hat{\sigma}^2$ holds because, by Theorems \ref{t:tailregeneration}
and \ref{t:regtimes}, the assumptions of the Donsker-Prohorov invariance principle are satisfied.

Now consider the random time change $\varphi_n(t) = k_{nt}/n$. We claim that
\be{eq:convphin}
\lim_{n\to \infty} \sup_{t \in [0,L]}\left|\varphi_n(t) - \frac{t}{\E^\um\left[\tau\right]} \right|
= 0 \quad \mathbb{P} \text{-a.s.} \quad \forall \, L > 0.
\ee
To prove \eqref{eq:convphin}, fix $\delta > \epsilon > 0$ and recall \eqref{defspeed}.
Reasoning as for \eqref{prLT2}, we see that
\begin{equation}
\label{eq:relationkandS}
\P \left( \exists \; t \in [\delta c,L] \colon \left| \varphi_n(t) - t/\iota \right| > \epsilon \right)
\le \P \left( \exists \; k \ge \delta'n \colon |S_k - k \iota| > \epsilon' k \right)
\end{equation}
for some $\delta', \epsilon' > 0$ and large enough $n$. By \eqref{prLT1}, the right-hand
side of \eqref{eq:relationkandS} is summable, and hence
\begin{equation}
\label{eq:stepconvphin}
\limsup_{n \to \infty} \sup_{t \in [0,L]}\left|\varphi_n(t) -t/\iota \right| \le 2 \delta \; \text{ a.s.}
\end{equation}
Since $\delta>0$ is arbitrary, \eqref{eq:convphin} follows. In particular, $\varphi_n$ converges
a.s.\ in the Skorohod topology to the linear function $t \mapsto t/\E^\um\left[\tau\right]$.

Define $Y^{(n)}_t = n^{-1/2} Y_{\lfloor nt \rfloor}$. With a time-change argument (see e.g.\
Billingsley \cite[Eqs.\ (17.7)--(17.9) and Theorem 4.4]{Bi}), we see that
 $(Y^{(n)}_{\varphi_n(t)})_{t \ge 0}$ converges weakly to a Brownian motion with variance
\begin{equation}\label{e:defsigma}
\sigma^2 = \frac{\E^\um \left[\left(X_\tau - \tau v \right)^2 \right]}{\E^\um[\tau]}.
\end{equation}
To extend this to $X$, note that, for any $T>0$,
\begin{align}
\label{supconv}
\sup_{0 \le t \le T} \left|\frac{X_{\lfloor nt \rfloor}
- \lfloor nt \rfloor v}{\sqrt{n}} - Y^{(n)}_{\varphi_n(t)} \right|
& \le \frac{1}{\sqrt{n}} \sup_{0 \le t \le T} \left( \left| X_{\lfloor nt \rfloor}
- X_{S_{k_{nt}}} \right| + v \left| S_{k_{nt}} - \lfloor nt \rfloor \right| \right) \nonumber\\
& \le \frac{|v| + 1}{\sqrt{n}}  \sup_{1\le k \le nT + 1} \left| \tau_k \right|,
\end{align}
which tends to $0$ a.s.\ as $n\to \infty$ by Theorems~\ref{t:tailregeneration} and \ref{t:regtimes}.
\end{proof}

\begin{remark}\label{r:continuity}
One may check that the statement of Theorem~\ref{t:tailregeneration} 
is uniform over compact intervals of the parameters $v_\circ$, $v_\bullet$, $\rho$.
Using this and the formulas \eqref{defspeed} and \eqref{e:defsigma},
it is possible to show continuity of $v$ and $\sigma$ in these parameters 
by approximating certain relevant observables of the system up to the regeneration time
by observables supported in finite space-time boxes. See e.g.\ Section 6.4 of \cite{dHodSa}.
\end{remark}

\subsection{Extensions}
\label{ss:regextensions}

As mentioned in Section~\ref{s:intro}, finding a regeneration structure is usually a delicate 
matter, as often one needs to rely on precise features of the model at hand. Approximate 
renewal schemes are more general, but do not usually give as much information as full 
regeneration.

Let us mention other examples of dynamic random environments where a renewal strategy 
can be found. For the simple symmetric exclusion process, such a renewal structure was 
developed in \cite{AvdSaVo}. There, the tail of the regeneration time is controlled by imposing 
a non-nestling condition on the random walk drifts. Using the techniques of this section, it 
would be possible to extend these results (i.e., obtain Theorem~\ref{thm:main}) to the nestling 
situation, provided one manages to prove the analogue of Theorem~\ref{thm:LD} for the 
exclusion process.

Another example where a regeneration strategy is useful is the independent renewal chain 
discussed in Section~\ref{ss:commentsreg}. Indeed, a regeneration time can be obtained 
as follows. Recall that, for large enough $\rho$, we obtain the ballisticity condition \eqref{e:LD} 
for some $v_\star >0$. Retaining the notation of Section~\ref{ss:morenotation} for $\bar{v}$, 
$\um(y)$, $R_k$ and $A^y$, we define
\begin{equation}
\mathcal{I} = \inf \{ k \in \N \colon\, A^{Y_{R_k}} \text{ occurs}\} \;\; 
\text{ and } \;\; \tau := R_{\mathcal{I}}.
\end{equation}
We may then verify that $\tau$ satisfies analogous properties as stated in 
Theorems~\ref{t:regeneration} and \ref{t:tailregeneration}. Hence, by the exact same arguments 
as in Section~\ref{ss:limittheorems}, Theorem~\ref{thm:main} holds also in this case.


The remainder of this paper consists of five appendices. All we have used so far is
Theorem~\ref{t:decouple} in Appendix~\ref{s:decouple} (recall Section~\ref{ss:space_time}),
which is a decoupling inequality, Lemma~\ref{l:ineq} in Appendix~\ref{s:inequalities}
(recall Section~\ref{ss:ldest}), which is a tail estimate, and Lemma~\ref{l:convrate}
in Appendix~\ref{s:conv_rate} (recall Section~\ref{ss:limittheorems}), which
is an estimate for independent random variables satisfying a certain tail assumption.
Appendices~\ref{s:simwithPoisson}--\ref{s:simulation}
are preparations for Appendix~\ref{s:decouple}.


\appendix


\section{Simulation with Poisson processes}
\label{s:simwithPoisson}

In this section we recall some results from Popov and Teixeira~\cite{PoTe} about how to
simulate random processes with the help of Poisson processes. Corollary~\ref{c:couplesystem}
will be used in Section~\ref{s:simulation} to prove a mixing-type result for a collection of
independent random walks (Lemma~\ref{l:couple} below).

Let $(\Sigma, \mathcal{B}, \mu)$ be a measure space, with $\Sigma$ a locally compact
Polish metric space, $\mathcal{B}$ the Borel $\sigma$-algebra on $\Sigma$, and
$\mu$ a Radon measure, i.e., every compact subset of $\Sigma$ has finite $\mu$-measure.
The set-up is standard for the construction of a Poisson point process on $\Sigma$.
To that end, consider the space of Radon point measures on $\Sigma \times \mathbb{R}_+$,
\begin{equation}
\label{e:M}
M = \left\{\m = \sum_{i \in \N} \delta_{(z_i, v_i)}\colon\, z_i \in \Sigma, v_i \in \mathbb{R}_+,
\m(K) < \infty\,\forall\, K \subseteq \Sigma \times \mathbb{R}_+ \text{ compact} \right\}.
\end{equation}
We can canonically construct a Poisson point process $\m$ on the measure space
$(M, \mathcal{M}, \mathbb{Q})$ with intensity $\mu \otimes \d v$, where $\d v$ is the
Lebesgue measure on $\mathbb{R}_+$. (For more details on this construction, see
e.g.\ Resnick~\cite[Proposition~3.6]{Re}.)

Proposition~\ref{p:simulate} below provides us with a way to simulate a random element of
$\Sigma$ by using the Poisson point process $\m$. Although the result is intuitive, we include
its proof for the sake of completeness.

\begin{proposition}
\label{p:simulate}
Let $g\colon\,\Sigma \to \mathbb{R}_+$ be a measurable function with $\int_\Sigma g(z) \mu(\d z)
= 1$. For $\m = \sum_{i \in \N} \delta_{(z_i, v_i)} \in M$, define
\begin{equation}
\xi = \inf \{ t \geq 0\colon\, \exists\,i \in \N \text{ such that } t g(z_i) \geq v_i\}
\end{equation}
(see Fig.~{\rm \ref{f:xi}}). Then, under the law $\mathbb{Q}$ of the Poisson point process $\m$,
\begin{enumerate}
\item[{\rm (1)}]
\label{e:io}
A.s.\ there exists a single $\hat{\imath} \in \N$ such that $\xi g(z_{\hat{\imath}}) = v_{\hat{\imath}}$.
\item[{\rm (2)}]
\label{e:xiio}
$(z_{\hat{\imath}}, \xi)$ has distribution $g(z) \mu(dz) \otimes \Exp(1)$.
\item[{\rm (3)}]
\label{e:mprime}
Let $\m' = \sum_{i \neq \hat{\imath}} \delta_{(z_i,v_i -\xi g(z_i))}$. Then $m'$ has the same
distribution as $\m$ and is independent of $(\xi, \hat{\imath})$.
\end{enumerate}
\end{proposition}

\begin{proof}
For measurable $A \subset \Sigma$, define the random variable
\begin{equation}
\xi^A = \inf \{ t \geq 0\colon\, \exists\, i \in \N \text{ such that } t \mathbbm{1}_A(z_i) g(z_i) \geq v_i\}.
\end{equation}
Elementary properties of Poisson point processes (see e.g.\ Resnick~\cite{Re}[(a-b), pp.\ 130])
yield that
\begin{display}
\label{e:xiA}
$\xi^A$ is exponentially distributed with parameter $\int_A g(z) \mu(dz)$, and if $A$
and $B$ are disjoint, then $\xi^A$ and $\xi^B$ are independent.
\end{display}
Property (1) follows from \eqref{e:xiA} because $\Sigma$ is separable and two independent
exponential random variables are a.s.\ distinct. Moreover, since
\begin{equation}
\mathbb{Q}[\xi \geq \alpha, z_{\hat{\imath}} \in A]
= \mathbb{Q}[\xi^{\Sigma \setminus A} > \xi^A \geq \alpha],
\end{equation}
Property (2) also follows from \eqref{e:xiA} by using elementary properties of the
minimum of independent exponential random variables. Thus it remains to prove Property
(3).

We first claim that, conditional on $\xi$, $\m'' = \sum_{i \neq \hat{\imath}} \delta_{(z_i,v_i)}$ is a
Poisson point process that is independent of $z_{\hat{\imath}}$ and 
has intensity measure $\mathbbm{1}_{\{v > \xi g(z)\}} [\mu(\d z) \otimes \d v]$. Indeed,
this is a consequence of the strong Markov property for Poisson point processes together
with the fact that $\{(z,v) \in \Sigma \times \mathbb{R}_+\colon\, v \leq \xi g(z)\}$ is a stopping
set (see Rozanov~\cite[Theorem~4]{Ro}).

Now, given $\xi$, $\m'$ is a mapping of $\m''$ (in the sense of Resnick~\cite[Proposition~3.7]{Re}).
This mapping pulls back the measure $\mathbbm{1}_{\{v > \xi g(z)\}}[\mu(\d z) \otimes \d v]$
to $\mu(\d z) \otimes \d v$. Since the latter distribution does not involve $\xi$, we obtain Property
(3).
\end{proof}

\begin{figure}
\centering \includegraphics[width = 0.65 \textwidth]{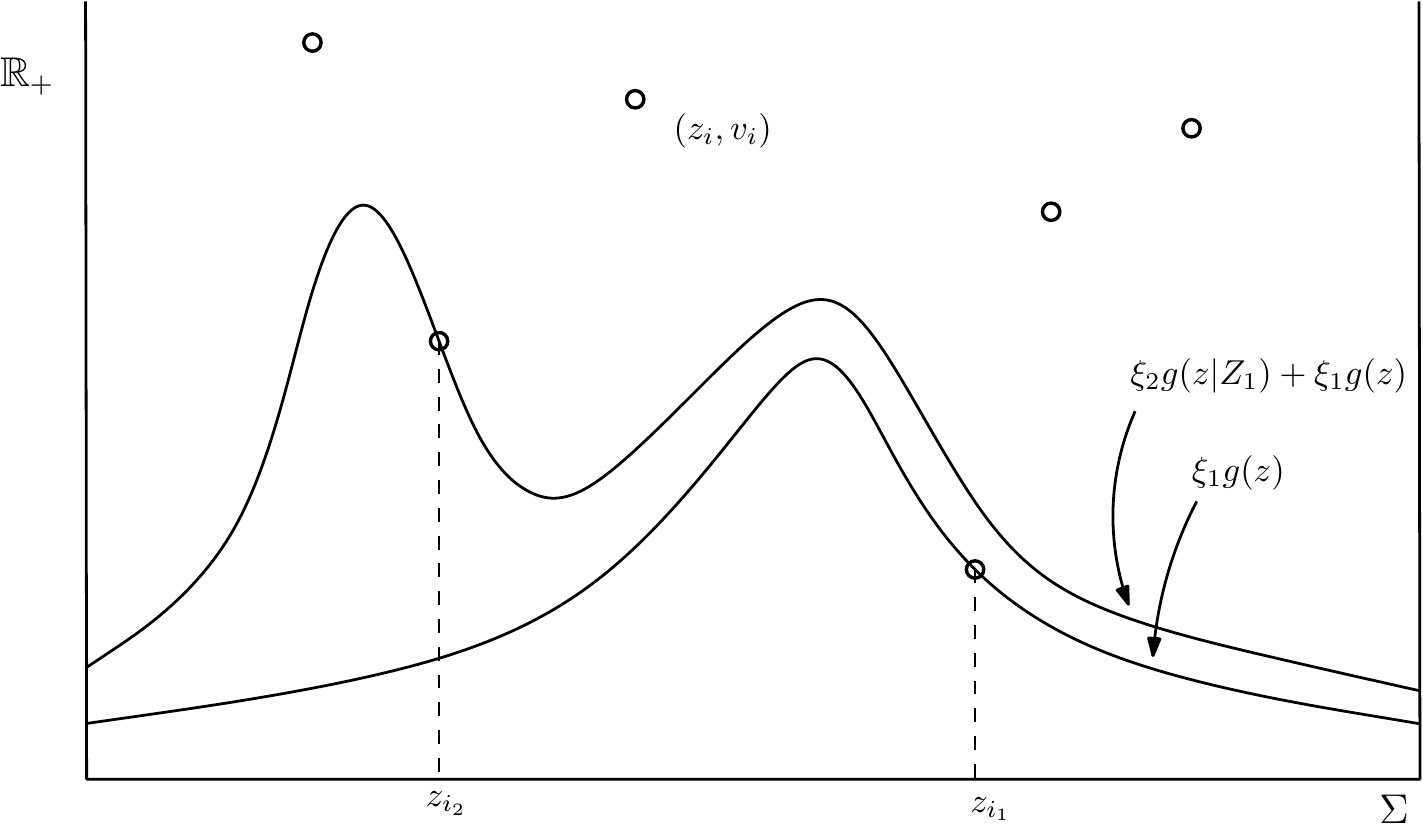}
\caption{An example illustrating the definition of $\xi$ and $\hat{\imath}$ in
Proposition~\ref{p:simulate}, and the definition of $\xi_1$, $Z_1$ and $\xi_2$, $Z_2$
in \eqref{e:xis}}.
\label{f:xi}
\end{figure}

In Proposition~\ref{p:poisson} below we use Proposition~\ref{p:simulate} to simulate a collection
$(Z_j)_{j\in\N}$ of independent random elements of $\Sigma$ using the single Poisson point
process $\m$ defined above. Formally, suppose that in some probability space $(M, \mathcal{M},
\mathcal{P})$ we are given a collection $(Z_j)_{j \in \N}$ of independent (not necessarily identically
distibuted) random elements of $\Sigma$ such that
\begin{equation}
\label{e:density2}
\text{the distribution of $Z_j$ is given by $g_j(z) \mu(dz)$, $j \in \N$}.
\end{equation}
In the same spirit as the definition of $\xi$ in Proposition~\ref{p:simulate}, we define what we call
the \emph{soft local time} $G=(G_j)_{j=1}^k$ associated with a sequence $(g_j)_{j=1}^k$ of
measurable functions:
\begin{equation}
\label{e:xis}
\begin{split}
& \xi_{1} = \inf \big\{ t \geq 0\colon\, tg_1(z_i) \geq v_i \text{ for at least one } i \in \N\big\},\\
& \quad G_{1}(z) = \xi_{1} \, g_1(z),\\
& \qquad \quad \vdots\\
& \xi_{k} =  \inf \big\{ t \geq 0\colon\, tg_k(z_i) + G_{k-1}(z_i) \geq v_i \text{ for at least } k
\text{ indices } i \in \N\big\},\\
& \quad G_{k}(z) = \xi_{1} \, g_1(z) + \dots + \xi_{k} \, g_k(z)
\end{split}
\end{equation}
(see Fig.~\ref{f:xi} for an illustration of this recursive procedure).

\begin{proposition}
Subject to {\rm (\ref{e:density2}--\ref{e:xis})},
\label{p:poisson}
\begin{align}
\label{e:ijk}
&(\xi_{j})_{j = 1}^{J} \text{ are i.i.d. } \mathrm{EXP}(1).\\
&\text{A.s.\ there is a unique $i_{J}$ such that $G_{J}(z_{i_{J}}) = v_{i_J}$}.\\
&(z_{i_{1}}, \dots, z_{i_{J}}) \overset d = (Z_1, \dots, Z_{J}).\\
&\text{$\m' = \;\; \sum_{\mathclap{i \not \in \{i_{1}, \dots, i_{J}\}}} \;\; \delta_{(z_i, v_i - G_{J}(z_i))}$
is distributed as $\m$ and is independent of the above.}
\end{align}
\end{proposition}

\begin{proof}
Apply Proposition~\ref{p:simulate} repeatedly, using induction on $J$.
\end{proof}

We close this section by exploiting the above construction to couple two collections
of independent random elements of $\Sigma$ using the same Poisson point process
as basis. The following corollary will be needed in Appendix~\ref{s:simulation}.

\begin{corollary}
\label{c:couplesystem}
Let $(g_j(\cdot))_{j = 1}^J$ be a family of densities with corresponding $\xi_{j}$, $G_{j}$,
$i_{j}$, $j = 1, \dots, J$, as in {\rm (\ref{e:xis}--\ref{e:ijk})}. Then, for any $\rho > 0$,
\begin{equation}
\label{Qineq}
\mathbb{Q} \left[ \sum_{j \leq J} \delta_{z_{i_{j}}} \geq  \sum_{i\colon v_i < \rho} \delta_{z_i} \right]
\geq \mathbb{Q} \big[ G_{J} \geq \rho \big].
\end{equation}
\end{corollary}

\noindent
Note that the right-hand side of \eqref{Qineq} only depends on the soft local time, which may
e.g.\ be estimated through large deviation bounds.


\section{Simulation and domination of particles}
\label{s:simulation}


\subsection{Simple random walk}
\label{ss:SRW}

In this section we collect a few facts about the heat kernel of random walks on $\Z$.
Let $p_n(x,x') = P_x(Z_n = x')$, $x,x'\in\Z$, where $P_x$ stands for the law of a lazy
simple random walk $Z_n$ on $\Z$ as defined in Section~\ref{s:intro}, i.e., $p_1(0,x) > 0$
if and only if $x \in \{-1,0,1\}$ and $p_1(0,1) = p_1(0,-1)$. Then there exists constants
$C, c>0$ such that the following hold for all $n \in \N$:
\begin{align}
\label{e:localclt}
& \sup_{x \in \mathbb{Z}}  p_n(0,x) \leq \frac{C}{\sqrt{n}}, \\
\label{e:SRW1}
& |p_n(0,x) - p_n(0,x')| \le \frac{C |x-x'|}{n} \;\;\; \forall \; x, x' \in \Z, \\
\label{e:SRW2}
& P_0(|Z_n| > \sqrt{n} \log n ) \le \, C e^{-c\log^{2} n}.
\end{align}
For \eqref{e:localclt}, see e.g.\ Lawler and Limic~\cite[Theorem 2.4.4]{LaLi}.
To get \eqref{e:SRW1}, use \cite[Theorem 2.3.5]{LaLi}, while \eqref{e:SRW2}
follows by an application of Bernstein's inequality.

The above observations will be used in the proof of Lemma~\ref{l:integration}
below, which deals with the integration of the heat kernel over an evenly
distributed cloud of sample points and is crucial in the proof of Theorem~\ref{t:decouple}
in Section~\ref{s:decouple}. In order to state this lemma, we need the following definitions.

\begin{definition}
\label{l:balanced}
(a) For $H\subset \Z$ and $L \in \N$,
we say that a collection of intervals $\{C_i\}_{i \in I}$ indexed by a subset $I \subset \mathbb{Z}$ is an $L$-paving of $H$
when $H \subset \cup_{i \in I} C_i$ and there is an $x \in \Z$ such that
\begin{equation}
\label{e:paving}
\{C_i\}_{i \in I} = \{[0,L) \cap \Z + Li + x \colon\,  i  \in I\}.
\end{equation}
(b) We say that a collection of points $(x_j)_{j \in J} \subset \mathbb{Z}$ is $\rho$-dense with
respect to the $L$-paving $\{C_i\}_{i \in I}$ when
\begin{equation}
\label{e:rhodense}
\#\{j \colon x_j \in C_i\} \ge \rho L \;\;\; \forall \; i \in I.
\end{equation}
\end{definition}

We know that $\sum_{x\in\Z} p_n(0,x)=1$. The next lemma approximates this normalization when
the sum runs over a dense collection $(x_j)_{j \in J}$.

\begin{lemma}
\label{l:integration}
Let $\{C_i\}_{i \in I}$ be an $L$-paving of $H \subset \Z$ and $(x_j)_{j \in J}$ be a $\rho$-dense collection
with respect to $\{C_i\}_{i \in I}$. Then, for all $n \in \N$,
\begin{equation}
\sum_{j \in J} p_n(0,x_j) \geq \rho \left\{ P_0\left(Z_n \in H \right) - \frac{cL \log n}{\sqrt{n}}\right\}.
\end{equation}
\end{lemma}

\begin{proof}
For each $i \in I$, choose $z_i \in C_i$ such that
\begin{align}
p_n(0,z_i) = \min_{x \in C_i} p_n(0,x).
\end{align}
Then we have
\begin{align}
\label{e:int1}
& \sum_{j \in J} p_n(0,x_j) \geq \sum_{i \in I} \sum_{j \colon x_j \in C_i} p_n(0,x_j)
\ge \sum_{i \in I} \rho L p_n(0,z_i) \nonumber\\
& \ge - \rho \sum_{i \in I} \sum_{x \in C_i} |p_n(0,x) -p_n(0,z_i)| + \rho P_0(Z_n \in H).
\end{align}
On the other hand, by \eqref{e:SRW1}--\eqref{e:SRW2} we have
\begin{align}\label{e:int2}
\sum_{i \in I} \sum_{x \in C_i} |p_n(0,x) -p_n(0,z_i)|
& \le 2 P_0(|Z_n| > \sqrt{n} \log n) + \sum_{|x| \le \sqrt{n} \log n} cL/n \nonumber\\
& \le c L \log n / \sqrt{n}
\end{align}
and the claim follows by combining \eqref{e:int1} and \eqref{e:int2}.
\end{proof}


\subsection{Coupling of trajectories}
\label{ss:couptraj}

Given a sequence of points $(x_j)_{j \in J}$ in $\mathbb{Z}$, let $(Z^j_n)_{n \in \Z_+}$, $j \in J$,
be a sequence of independent simple random walks on $\mathbb{Z}$ starting at $x_j$, and
let $\bigotimes_{j \in J} P_{x_j}$ denote their joint law. The next lemma, which will be needed
in Appendix~\ref{s:decouple}, provides us with a way to couple the positions of these random
walks at time $n$ with a Poisson point process on $\mathbb{Z}$. This lemma is similar in
flavor to \cite[Proposition 4.1]{PeSiSoSt}.

\newconstant{c:couple_exp}
\newconstant{c:couple}

\begin{lemma}
\label{l:couple}
Let $(x_j)_{j \in J} \subset \mathbb{Z}$ be $\rho$-dense with respect to the $L$-paving
$\{C_i\}_{i \in I}$ of $H \subset \Z$.
Then for any $\rho' \leq \rho$ there exists a coupling $\mathbb{Q}$ of
$\otimes_{j \in J} P_{x_j}$ and the law of a Poisson point process $\sum_{j' \in J'}
\delta_{Y_{j'}}$ on $\mathbb{Z}$ with intensity $\rho'$ such that
\begin{equation}
\mathbb{Q} \left[ \1{H'} \sum_{j \in J} \delta_{Z^j_n} \geq \1{H'}
\sum_{j' \in J'} \delta_{Y_{j'}} \right]
\geq 1 - |H'| \, \exp \left\{ - (\rho - \rho')L
+ \left(\frac{\useconstant{c:couple_exp} \rho L^2 \log n}{\sqrt{n}} \right) \right\}
\end{equation}
for all $H' \subset \Z$ such that $ \{z \in \Z \colon \dist(z,H') \le n\} \subset H$
and all $n \geq \useconstant{c:couple} L^2$.
\end{lemma}

\begin{proof}
By Corollary~\ref{c:couplesystem}, there exists a coupling $\mathbb{Q}$ such that
\begin{equation}
\label{e:coupleGG}
\mathbb{Q} \left[ \1{H'} \sum_{j \in J} \delta_{Z^j_n} \geq \1{H'}  \sum_{j' \in J'} \delta_{Y_{j'}} \right]
\geq \mathbb{Q} \big[ G_J(z) \geq \rho'\,\,\forall\,z \in H'\big],
\end{equation}
where $G_J(z) = \sum_{j \in J} \xi_j \, p_{n}(x_j, z)$ with $(\xi_j)_j$ i.i.d.\ $\mathrm{EXP}(1)$
random variables. We will estimate the right-hand side of \eqref{e:coupleGG} using concentration
inequalities. First, noting that $P_z(Z_n \in H) = 1$ for any $z \in H'$,
we use Lemma~\ref{l:integration} to estimate
\begin{equation}
\label{e:EGJ}
E^{\mathbb{Q}}[G_J(z)] = \sum_{j \in J} p_{n}(z, x_j)
\geq \rho \Big( 1 - \frac{cL \log n}{\sqrt{n}} \Big), \qquad z \in H'.
\end{equation}
Next, estimate
\begin{align}
\nonumber
&\mathbb{Q} \big[\exists\, z \in H'\colon\,G_J(z) < \rho'\big]
\leq |H'| \sup_{z \in H'}\mathbb{Q} [ G_J(z) < \rho']\\
& \label{e:QGJz}
\qquad \leq |H'| \, e^{\rho'L} \sup_{z \in H'}E^{\mathbb{Q}} \big[ \exp\{ - L G_J(z) \} \big].
\end{align}
Now note that, for any $z \in \mathbb{Z}$,
\begin{equation}
E^{\mathbb{Q}} \big[ \exp\{ - L G_J(z) \} \big] = \smash{\prod_{j \in J}}
E^{\mathbb{Q}} \big[ \exp\{ - L \xi_j p_n(x_j,z) \} \big]
= \prod_{\smash{j \in J}} \Big(1 + L p_n(x_j,z) \Big)^{-1}.
\end{equation}
Assuming that $n \geq (c L)^2 \vee e$, we can use \eqref{e:localclt} to obtain that
\begin{equation}
\label{e:localclt2}
\sup_{x \in \mathbb{Z}} L p_n(0,x) \leq \frac{cL}{\sqrt{n}} \leq \tfrac12.
\end{equation}
A simple Taylor expansion shows that $\log(1 + u) \geq u - u^2$ for every $|u| \leq \tfrac12$.
Hence $1 + u \geq \exp \{u (1 - u)\}$ and
\begin{equation}
\begin{array}{e}
\prod_{j \in J}(1 + L p_n(z,x_j))^{-1}
& \leq & \prod_{j \in J} \exp \big\{ - L p_n(z,x_j) \big(1 - L p_n(z,x_j) \big) \big\}\\
& \leq & \exp \Big\{ - \sum_{j \in J} L p_n(z,x_j) \big( 1 - \sup_{x \in \mathbb{Z}} L p_n (0,x)  \big) \Big\}\\
& \leq & \exp \Big\{ - \rho L \big(1 - \tfrac{cL \log n}{\sqrt{n}} \big) \big( 1 - \tfrac{cL}{\sqrt{n}}  \big) \Big\}\\
&\leq & \exp \Big\{ - \rho L \big(1 - \tfrac{2cL \log n}{\sqrt{n}} \big) \Big\}.
\end{array}
\end{equation}
where the third inequality uses \eqref{e:EGJ} and \eqref{e:localclt2}. Inserting this estimate
into \eqref{e:QGJz}, we get the claim.
\end{proof}


\section{Decoupling of space-time boxes}
\label{s:decouple}

In this section we prove a decoupling inequality on two disjoint boxes in the space-time
plane $\Z_+ \times \Z$.


\subsection{Correlation}
\label{ss:correlation}

Intuitively, if two events depend on what happens at far away times, then they must be
close to independent due to the mixing of the dynamics. This is made precise in the
following theorem.

\newconstant{c:largendecouple}

\begin{theorem}
\label{t:decouple}
Let $B = ([a, b] \times [n, n']) \cap \Z^2$ be a space-time box as in Fig.~{\rm \ref{f:light_cone}}
for some $n \geq \useconstant{c:largendecouple}$, and let $D = \Z \times \Z_-$ be the space-time
lower halfplane. Recall Definitions~{\rm \ref{d:monotone}--\ref{d:support}}, and assume that
$f_1\colon\,\Omega \to [0,1]$ and $f_2\colon\,\Omega \to [0,1]$ are non-increasing random
variables with support in $D$ and $B$, respectively. Then, for any $\rho \geq 1$,
\begin{equation}
\label{e:dec}
\mathbb{E}^{\rho(1+n^{-1/16})}[f_1 f_2] \leq \mathbb{E}^{\rho(1+n^{-1/16})}[f_1] \,\,
\mathbb{E}^{\rho}[f_2] + c \big(\textnormal{per}(B) + n\big) \,e^{-c \rho n^{1/8}},
\end{equation}
where $\textnormal{per}(B)$ stands for the perimeter of $B$.
\end{theorem}

Note that, by the FKG-inequality, we have $\mathbb{E}^{\rho}[f_1 f_2] \geq \mathbb{E}^{\rho}[f_1]\,
\mathbb{E}^{\rho}[f_2]$. Thus, the bound in \eqref{e:dec} shows that $f_1,f_2$ are almost
uncorrelated.

To prove Theorem~\ref{t:decouple}, we need the following definition. For a box $B = ([a,b]
\times [n, n']) \cap \Z^2$ in the space-time upper halfplane $\Z \times \Z_+$, let $\mathcal{C}(B)$
be the cone associated with $B$, defined as (see Fig.~\ref{f:light_cone})
\begin{equation}
\label{e:cone}
\mathcal{C}(B) = \bigcup_{k = 0}^\infty \big([a-k, b+k] \times \{n'-k\}\big).
\end{equation}
This cone can be interpreted as the set of points that can reach $B$ while traveling at
speed at most one, and encompasses every space-time point that can influence the state of $B$.

Given a box $B$ and a halfplane $D$ as in Theorem~\ref{t:decouple}, we denote by $H$
and $H'$ the separating segments (see Fig.~\ref{f:light_cone})
\begin{equation}
\label{e:Hline}
H' = \mathcal{C}(B) \cap \big( \Z \times \{n\} \big), \qquad
H  = \mathcal{C}(B) \cap \big( \Z \times \{0\} \big).
\end{equation}
The next lemma states a Markov-type property.

\begin{figure}[htbp]
\centering
\begin{tikzpicture}[scale=1]
\draw[fill, color=gray!20!white] (0,0) rectangle (12,2);
\draw[dashed] (0,2) -- (12,2);
\draw[fill, color=gray!40!white] (0,0) -- (5.5,5.5) -- (6.5,5.5) -- (12,0) -- (0,0);
\draw (0,0) -- (5.5,5.5) -- (6.5,5.5) -- (12,0);
\draw[thick] (5.5,4.5) rectangle (6.5,5.5);
\draw[thick] (2,2) -- (10,2);
\draw[thick] (2.1,2.2) -- (2,2.2) -- (2,1.8) -- (2.1,1.8);
\draw[thick] (9.9,2.2) -- (10,2.2) -- (10,1.8) -- (9.9,1.8);
\foreach \x in {2,...,10} {\draw[thick] (\x,1.8) -- (\x,2.2);}
\draw (2,3.5) edge[out=330,in=120,->] (3.5,2);
\node[left] at (2,3.5) {$H$};
\draw[thick] (4.5,4.5) -- (7.5,4.5);
\draw[thick] (4.6,4.7) -- (4.5,4.7) -- (4.5,4.3) -- (4.6,4.3);
\draw[thick] (7.4,4.7) -- (7.5,4.7) -- (7.5,4.3) -- (7.4,4.3);
\draw[dashed] (0,4.5) -- (4.5,4.5);
\node[left] at (0,4.5) {$n$};
\draw[dashed] (0,5.5) -- (5.5,5.5);
\node[left] at (0,5.5) {$n'$};
\draw (4,5) edge[out=0,in=120,->] (5,4.5);
\node[left] at (4,5) {$H'$};
\node[below right] at (0,2) {$D$};
\node[left] at (0,2) {$0$};
\node[below left] at (6.5,5.5) {$B$};
\end{tikzpicture}
\caption{Box $B$ and lower halfplane $D$ as in Theorem~\ref{t:decouple}.
The dark gray area corresponds to the cone associated with $B$, containing
the separating segments $H$ and $H'$.}
\label{f:light_cone}
\end{figure}
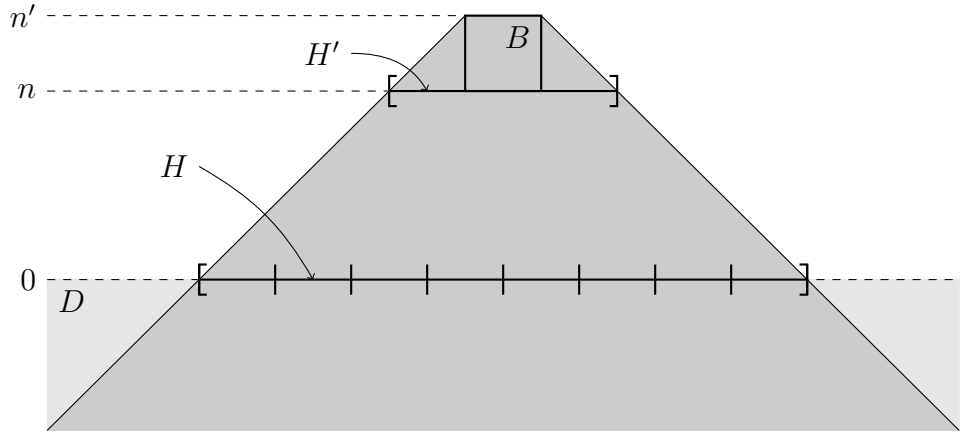


\begin{lemma}
\label{l:markov_cone}
Let $B = [a,b] \times [n, n'] \cap \Z^2$, and let $\mathcal{C}(B)$ and $H$ be as in
{\rm (\ref{e:cone}--\ref{e:Hline})}. Then, for any function $f\colon\, \Omega \to \R$
with support in $B$,
\begin{equation}
\label{e:markov_cone}
\mathbb{E}\Big[f \Big| N(y), y \in D \Big] = \mathbb{E} \Big[f \Big| N(y), y \in H \Big],
\end{equation}
where $N(y)$ is the number of trajectories crossing $y$ (recall \eqref{e:defIPS}).
\end{lemma}

\begin{proof}
Since $y \in \cT$ if and only if $N(y) \ge 1$, $f$ is a function of $(U_y, N(y))_{y \in B}$.
Noting that $(N(y))_{y \in B}$ is a function of $(N(y))_{y \in H}$ and $(S^{y,i}_n)_{y \in H, i \in \N, n \in \Z_+}$
only, we get the claim.
\end{proof}


\subsection{Proof of Theorem~\ref{t:decouple}}
\label{ss:PrThDecouple}

In the following we will abuse notation by writing $H$,$H'$ to denote
also the projection on $\Z$ of these sets.
We start by choosing an $L$-paving $\{I_j\}_{j \in J}$ of $H$, composed
of segments of length $L = \lfloor n^{1/4} \rfloor$ (the choice of exponent $1/4$ is arbitrary:
any choice in $(0,\tfrac12)$ will do). Note that
\begin{equation}
\label{e:cover_H}
\text{no more than $1 + (\textnormal{per}(B) + 2n)/L$ such segments are needed to cover $H$,}
\end{equation}
i.e., we may take $|J| \le 1 + (\textnormal{per}(B) + 2n)/L$. Note also that $|H'| \le \textnormal{per}(B)+1$.

We take $\useconstant{c:largendecouple}$ so large that
\begin{equation}
\label{e:hypot_constant}
n > \useconstant{c:largendecouple} \;\; \Rightarrow \;\; L > n^{1/4}/2, \;\;
n > \useconstant{c:couple}L^2 \; \text{ and } \; n^{1/8} > 16 \useconstant{c:couple_exp} \log n.
\end{equation}

Abbreviate $\rho_r = \rho (1 + rn^{-1/16}) $ for $r \in [0,1]$. Note that the two densities
appearing in the statement of Theorem~\ref{t:decouple} are $\rho = \rho_0$ and $\rho(1 + n^{-1/16})
= \rho_1$. We introduce the events
\begin{equation}
\label{e:event_D}
\mathcal{E} = \Big\{ \text{at least $\rho_{\mtiny{1/2}} L$ trajectories cross $H \cap I_j$ for all $j \in\N$} \Big\}.
\end{equation}
Then, by Lemma~\ref{l:markov_cone} and $f_1, f_2 \leq 1$,
\begin{equation}
\label{e:f1f2}
\begin{array}{e}
\mathbb{E}^{\rho_1}[f_1 f_2]
& \leq & \mathbb{P}^{\rho_1} [\mathcal{E}^c]
+ \mathbb{E}^{\rho_1} \Big[f_1 1_{\mathcal{E}} \mathbb{E}^{\rho_1}
\big[f_2 \big| N(y),y \in D \big]\Big]\\[0.3cm]
& = & \mathbb{P}^{\rho_1} [\mathcal{E}^c]
 + \mathbb{E}^{\rho_1} \Big[f_1 1_{\mathcal{E}} \mathbb{E}^{\rho_1}
 \big[f_2 \big| N(y),y \in H \big] \Big].
\end{array}
\end{equation}
To estimate the first term in the right-hand side of \eqref{e:f1f2}, we do a moderate deviation
estimate. For every $\theta>0$, by \eqref{e:cover_H},
\begin{equation}
\label{e:exp_cheby}
\begin{split}
\mathbb{P}^{\rho_1}[\mathcal{E}^c] & \leq \big| \{j\in\N\colon\,I_j
\cap H \neq \emptyset\} \big| \,\mathbb{P}^{\rho_1}
\Big[\sum_{y \in I_j} N(y) < \rho_{\mtiny{1/2}}L\Big]\\
& \leq \Big(1 + L^{-1}(\textnormal{per}(B) + 2n) \Big)
\exp \Big\{ \rho_{\mtiny{1/2}}L\theta - \rho_1 L (1-e^{-\theta}) \Big\},
\end{split}
\end{equation}
since $N(y)$ has law Poisson($\rho_1$).
Since $1-e^{-\theta} \ge \theta - \theta^2$, we have
\begin{equation}
\begin{split}
\exp \Big\{ \rho_{\mtiny{1/2}}L\theta -
& \rho_1 L (1-e^{-\theta}) \Big\} \leq
\exp \Big\{ \rho_{\mtiny{1/2}} L \theta - \rho_1 L \theta + \rho_1 L \theta^2 \Big\}\\
& = \exp \Big\{ L\theta\rho_{\mtiny{0}}
\big[ (1 + \tfrac 12 n^{-1/16}) - (1 + n^{-1/16}) + 2 \theta \big] \Big\}\\
& = \exp \big\{ L\theta \rho_{\mtiny{0}} [-\tfrac 12 n^{-1/16} + 2\theta]\big\}.
\end{split}
\end{equation}
Choosing $\theta = \tfrac18 n^{-1/16}$, we see that the above equals
\begin{equation}
\label{e:exp_taylor}
\begin{split}
\exp \{ - \tfrac14 L \theta \rho_{\mtiny{0}} n^{-1/16} \}
= \exp \{ - \tfrac{1}{32} \rho_{\mtiny{0}} L n^{-1/8} \}
\leq \exp \{ - \tfrac{1}{64} \rho_{\mtiny{0}} n^{1/8}\}.
\end{split}
\end{equation}
where the second inequality uses \eqref{e:hypot_constant}.

Next, we employ Lemma~\ref{l:couple} to estimate the second term in the right-hand side
of \eqref{e:f1f2}. To that end, note that on the event $\mathcal{E}$ the collection of points
that hit $H$ is $\rho_{\mtiny{1/2}}$-dense with respect to the $L$-paving $\{I_j\}_{j\in J}$.
Moreover, $\{z \in \Z \colon \dist(z,H') \le n \} \subset H$. Therefore, applying Lemma~\ref{l:couple}
with the densities $\rho_{\mtiny{0}} \le \rho_{\mtiny{1/2}}$ (recall \eqref{e:hypot_constant}),
we have
\begin{align}
\label{e:apply_couple}
\mathbbm{1}_{\mathcal{E}} \mathbb{E}^{\rho_1} \big[ f_2  \mid  N(y), y \in H \big]
 & \leq \mathbb{E}^{\rho_{\mtiny{0}}} \big[f_2\big]
+ |H'| \exp \Big\{ - (\rho_{\mtiny{1/2}} - \rho_{\mtiny{0}})L
+ \frac{\useconstant{c:couple_exp}\rho_{\mtiny{1/2}}L^2 \log n}{\sqrt{n}} \Big\}\nonumber\\
& \leq \mathbb{E}^{\rho_{\mtiny{0}}} \big[f_2\big]
+ |H'| \exp \big\{ - (\tfrac{\rho_{\mtiny{0}}}{2}) n^{-1/16} L
+ 2\useconstant{c:couple_exp} \rho_{\mtiny{0}} \log n \big\}.
\end{align}
By \eqref{e:hypot_constant}, we have
\begin{equation}
\text{r.h.s.} \eqref{e:apply_couple} \leq  \mathbb{E}^{\rho_{\mtiny{0}}} \big[f_2\big]
+ |H'| \; e^{ - \tfrac18 \rho_{\mtiny{0}} n^{1/8}}.
\end{equation}

Combine (\ref{e:f1f2}--\ref{e:exp_cheby}), to obtain
\begin{equation}
\label{e:final_f1f2}
\begin{split}
\mathbb{E}^{\rho_1}[f_1 f_2]
& \leq c \Big(1 + L^{-1}(\textnormal{per}(B) + 2n)\Big)\,e^{- c \rho_{\mtiny{0}} n^{1/8}}\\
& \quad + \mathbb{E}^{\rho_1}[f_1] \Big( \mathbb{E}^{\rho_{\mtiny{0}}}[f_2]
+ |H'| \, e^{- c \rho_{\mtiny{0}} n^{1/8}} \Big)\\
& \leq \mathbb{E}^{\rho_{1}}[f_1] \,\mathbb{E}^{\rho_{\mtiny{0}}}[f_2]
+ c(\textnormal{per}(B) + n)\, e^{ - c \rho_{\mtiny{0}} n^{1/8}},
\end{split}
\end{equation}
which yields the claim in \eqref{e:dec}.

\begin{remark}
It is important that the constants in Theorem {\rm \ref{t:decouple}} do not depend on $\rho$, in
accordance with our convention. This is crucial for our proof in Section~{\rm \ref{s:renorm}}
to work. Moreover, by translation invariance we can apply the result for general boxes $B$ and
half-spaces $D$ as long as their vertical distance is at least $\useconstant{c:largendecouple}$.
\end{remark}

\begin{remark}
\label{r:lazy2}
The proof of Theorem~{\rm ~\ref{t:decouple}} holds almost without modification
in the more general case of aperiodic symmetric random walks with bounded steps.
Indeed, since properties \eqref{e:localclt}--\eqref{e:SRW2} still hold in this case,
the proofs of Lemmas~{\rm \ref{l:integration}--\ref{l:couple}} immediately extend.
The only change that is needed is in the definition of $H$, which must be made
larger according to the bound on the size of the random walk steps. 

\noindent
In the case of bipartite random walks, however, the aforementioned lemmas are no longer
true as stated. Indeed, \eqref{e:SRW1} does not hold for such walks. Nonetheless,
Theorem~{\rm \ref{t:decouple}} is still true in this situation. The proof can be adapted
through the following two steps:
\begin{enumerate}
\item[(1)]
In the statement of Lemmas~{\rm \ref{l:integration}--\ref{l:couple}}, we suppose
that the collection $(x_j)_{j \in J}$ is dense in both of the sets $C_i \cap 2\Z$ and
$C_i \cap \Z \setminus 2\Z$. Since \eqref{e:SRW1} still holds when $x,x'$ have
the same parity, and \eqref{e:localclt} and \eqref{e:SRW2} are still valid, the proofs
of the lemmas go through with this modification.
\item[(2)]
In the proof of Theorem~{\rm \ref{t:decouple}}, we modify $\mathcal{E}$ to be the
event where enough trajectories cross both of the sets $H \cap I_i \cap 2\Z$ and
$H \cap I_i \cap \Z \setminus 2\Z$, which allows us to use
Lemmas~{\rm \ref{l:integration}--\ref{l:couple}} with the new statements.
\end{enumerate}

\noindent
In the case of continuous-time symmetric random walks with bounded jumps,
\eqref{e:localclt}--\eqref{e:SRW2} and Lemma~\ref{l:integration} remain true.
However, the random walk is no longer almost surely Lipschitz, and this property 
is used in Lemmas~\ref{l:couple} and \ref{l:markov_cone}, Theorem~\ref{t:decouple}
and in several other places throughout the paper. Nonetheless, the random walk is 
still Lipschitz with high probability, and this is enough to adapt all arguments to this 
case.
\end{remark}


\section{Tail estimate}
\label{s:inequalities}

The following lemma consists of a simple estimate on the tail of an infinite series,
used in \eqref{e:PBkc}. Its proof is inspired by similar arguments in M\"orters and
Peres~\cite[Lemma~12.9, p.\ 349]{MoPe}.

\begin{lemma}
\label{l:ineq}
Let $a$ be a positive integer. Then, for all $\beta > 0$ there exists a $c = c(\beta)$ such that
\begin{equation}
\sum_{l>a} l^{\beta} e^{-\log^{3/2}l} \leq ca^{\beta+1}e^{-\log^{3/2}a}.
\end{equation}
\end{lemma}

\begin{proof}
It is enough to consider the case $a \ge \alpha_0 = e^{(\beta+1)^2}$.
If $g(x) = x^\beta \exp\{-\log^{3/2}x\}$, then
$g'(x)=x^{\beta-1}\exp\{-\log^{3/2}x\}[\beta - \tfrac32\sqrt{\log x}] < 0$ for all $x \ge \alpha_0$ because
$\alpha_0 > e^{(4/9) \beta^2}$. It follows that
\begin{equation}
\sum_{l>a} l^{\beta} e^{-\log^{3/2}l} \leq \int_{a}^{\infty} x^{\beta} e^{-\log^{3/2}x}dx.
\end{equation}
Let
\begin{equation}
\begin{aligned}
D &= \max\left\{2\int_{\alpha_0}^{\infty}x^\beta e^{-\log^{3/2}x} dx, \frac{4}{\beta+1}\right\},\\
f(a) &= D a^{\beta+1} e^{-\log^{3/2}a}-\int_{a}^{\infty}x^\beta e^{-\log^{3/2}x}dx.
\end{aligned}
\end{equation}
We claim that $\lim_{a\to\infty} f(a) = 0$, $f(\alpha_0) > 0$ and $f'(a)<0$ for all $a\geq \alpha_0$.
The first claim is immediate. The second claim follows from
\begin{equation}
\begin{split}
f(\alpha_0)
&= De^{{\log{\alpha_0}\{(\beta+1)-\sqrt{\log{\alpha_0}} \} }}
-\int_{\alpha_0}^{\infty}x^\beta e^{-\log^{3/2}x}dx\\
&= D  -\int_{\alpha_0}^{\infty}x^\beta e^{-\log^{3/2}x}dx >0,
\end{split}
\end{equation}
where the last inequality follows from $D \geq 2\int_{\alpha_0}^{\infty}x^\beta e^{-\log^{3/2}x} dx$.
To get the third claim, note that
\begin{equation}
f'(a) =  a^\beta e^{-\log^{3/2}a}\, [(\beta+1) D - \tfrac32 D\sqrt{\log{a}} +1]
\end{equation}
and, for $a \geq \alpha_0$, we have that $[(\beta+1)D - \tfrac32D\sqrt{\log{a}} +1] \leq -\tfrac12
D (\beta+1) + 1 \leq -1$, where the last inequality follows from $D \geq {4}/({\beta+1})$.

In particular, $f(a) \ge 0$ for all $a \geq \alpha_0$, and hence
\begin{equation}
\sum_{l>a} l^{\beta} e^{-\log^{3/2}l} \leq \int_{a}^{\infty} x^{\beta}
e^{-\log^{6/5}x}\,dx \leq Da^{\beta+1} e^{-\log^{3/2}a},
\end{equation}
which concludes the proof.
\end{proof}


\section{Rate of convergence}
\label{s:conv_rate}

In this section we state and prove a basic fact about independent random variables that
is used in the proof of Theorem~\ref{thm:SLLN+CLT+LD}(c) in Section~\ref{ss:limittheorems}.

\begin{lemma}
\label{l:convrate}
Let $X_i$, $i \in \N$, be independent random variables with joint law $\P$ such that
\begin{equation}
\label{e:convrate1}
\E[X_i]=0 \; \forall \; i \in \N, \qquad \quad \sup_{i \in \N}
\E\left[ \exp\{K (\log^+ |X_i|)^\gamma \} \right] < \infty
\end{equation}
for some $K>0$ and $\gamma > 1$, where $\log^+ x = \max(\log x, 0)$. Then, for all
$\varepsilon > 0$, there exists a $c>0$ such that
\begin{equation}
\label{e:convrate2}
\P \left( \exists \; k \ge n \colon \left|\sum_{i=1}^k X_i \right| > \varepsilon k \right)
\le c^{-1} e^{-c \log^\gamma n} \;\; \forall \; n \in \N.
\end{equation}
\end{lemma}

\begin{proof}
Let
\begin{equation}
\label{prconvrate1}
X_i^{(n)} = X_i \mathbbm{1}_{\{|X_i| \le \sqrt{n}\}} - \E[X_i \mathbbm{1}_{\{|X_i| \le \sqrt{n}\}}].
\end{equation}
We claim that
\begin{equation}
\label{prconvrate2}
\P \left ( \exists \; k \in\N  \colon \left| \sum_{i=1}^k X_i - X_i^{(n)} \right| > \varepsilon k \right)
\le C e^{-c \log^\gamma n}.
\end{equation}
Indeed, since $\E[X_i] = 0$, setting
\begin{equation}
\label{prconvrate3}
X_i^{>(n)} = X_i \mathbbm{1}_{\{|X_i| > \sqrt{n}\}} - \E[X_i \mathbbm{1}_{\{|X_i| > \sqrt{n}\}}]
\end{equation}
we have $X_i - X^{(n)}_i = X^{>(n)}_i$ for all $i \in \N$. Moreover, by the Marcinkiewicz-Zygmund
and Minkowski inequalities,
\begin{equation}
\label{prconvrate4}
\E \left[  \left| \sum_{i=1}^k X_i^{>(n)} \right|^4\right]^{\frac12}
\le C \sum_{i=1}^k \E[|X_i^{>(n)}|^4]^{\frac12}
\le C k\, e^{-c \log^\gamma n}
\end{equation}
by \eqref{e:convrate1}. Therefore
\begin{equation}
\label{prconvrate5}
\P\left( \exists \; k \in \N \colon \left| \sum_{i=1}^k X_i^{>(n)} \right| > \varepsilon k \right)
\le C e^{-c \log^\gamma n} \sum_{k=1}^\infty  k^{-2}
\end{equation}
and \eqref{prconvrate2} follows. Thus, it suffices to prove \eqref{e:convrate2} for $X^{(n)}_i$.
To that end, we use Bernstein's inequality to write
\begin{equation}
\label{prcovrate6}
\P \left( \left| \sum_{i=1}^k X_i^{(n)} \right| > u \right) \leq 2 \exp \left\{ -c\, \frac{u^2}{k+ \sqrt{n} u} \right\}.
\end{equation}
Taking $u= \varepsilon k$, we obtain
\begin{align}
\label{prconvrate7}
\P \left( \exists\; k \ge n \colon \left| \sum_{i=1}^k X_i^{(n)} \right| > \varepsilon k \right)
\le 2 \sum_{k=n}^\infty e^{ -\frac{c}{\sqrt{n}} k}
= \frac{2e^{-c\sqrt{n}}}{1 - e^{ -\frac{c}{\sqrt{n}}}}
\le C e^{ -c \sqrt{n}},
\end{align}
which concludes the proof.
\end{proof}



\begin{thebibliography}{99}

\bibitem{AvdHoRe11}
L.~Avena, F.~den Hollander and F.~Redig,
Law of large numbers for a class of random walks in dynamic random environments,
Electron.\ J.\ Probab.\ 16 (2011) 587--617.

\bibitem{AvdSaVo}
L.\ Avena, R.S.\ dos Santos and F.\ V\"ollering,
Transient random walk in symmetric exclusion: limit theorems and an
Einstein relation,
ALEA (Lat.\ Am.\ J.\ Probab.\ Math.\ Stat.) 10 (2013) 693--709.

\bibitem{BR}
J.\ B\'erard and A.F.\ Ram\'irez,
Fluctuations of the front in a one dimensional model for the spread of an infection,
preprint, arXiv:1210.6781 [math.PR]. 

\bibitem{BhZtn}
A.\ Bandyopadhyay and  O.\ Zeitouni, 
Random Walk in Dynamic Markovian Random Environment,
ALEA (Lat.\ Am.\ J.\ Probab.\ Math.\ Stat.) 1 (2006) 205--224.

\bibitem{Bi}
P.\ Billingsley,
\emph{Convergence of Probability Measures},
Wiley, New York, 1968.

\bibitem{BiCeDeGa13}
M.~Birkner, J.~\v{C}ern\'y, A.~Depperschmidt and N.~Gantert,
Directed random walk on the backbone of an oriented percolation cluster,
Electron.\ J.\ Probab.\ 18 (2013).

\bibitem{BoCoFrPe05}
C.\ Boldrighini, G.\ Cosimi, S.\ Frigio and A.\ Pellegrinotti, 
Computer simulations for some one-dimensional models of random walks in fluctuating random environment, 
J.\ Stat.\ Phys 121 (2005) 361--372.

\bibitem{BoMiPe97}
C.\ Boldrighini, R.A.\ Minlos and A.\ Pellegrinotti,
Almost-sure central limit theorem for a Markov model of random walk in dynamical random environment, 
Probab\ Theory Related Fields 109 (1997) 245--273.

\bibitem{BoMiPe00}
C.\ Boldrighini, R.A.\ Minlos and A.\ Pellegrinotti, 
Random walk in a fluctuating random environment with Markov evolution. 
in: \textit{On Dobrushin's way. From probability theory to statistical physics}, 
volume 198 of Amer.\ Math.\ Soc.\ Transl.\ Ser.\ 2, 13--35. 
Amer.\ Math.\ Soc.\, Providence, RI, 2000.

\bibitem{BoMiPe04}
C.\ Boldrighini, R.A.\ Minlos and A.\ Pellegrinotti, 
Random walks in quenched i.i.d.\ space-time random environment are always a.s.\ diffusive, 
Probab.\ Theory Related Fields 129 (2004) 133--156.

\bibitem{BoMiPe09}
C.~Boldrighini, R.A.~Minlos and A.~Pellegrinotti, 
Discrete-time random motion in a continuous random medium,
Stoch.\ Proc.\ Appl.\ 119 (2009) 3285--3299. 

\bibitem{dkl}
D.\ Dolgopyat, G.\ Keller and C.\ Liverani, 
Random walk in Markovian environment,
Ann.\ Probab.\ 36 (2008) 1676--1710.

\bibitem{dk}
D.\ Dolgopyat and C.\ Liverani,  
Non-perturbative approach to random walk in Markovian environment, 
Electron.\ Commun.\ Probab.\ 14 (2009) 245--251. 

\bibitem{Ha}
T.\ Harris,
Diffusion with ``collision'' between particles,
J.\ Appl.\ Probab.\ 2 (1965) 323--338.

\bibitem{dHodSa}
F.\ den Hollander and R.\ dos Santos,
Scaling of a random walk on a supercritical contact process,
Ann.\ Inst.\ H.\ Poincar\'e Probab.\ Statist.\ 50 (2014) 1276--1300.

\bibitem{dHoKeSi}
F.\ den Hollander, H.\ Kesten and V.\ Sidoravicius,
Random walk in a high density dynamic random environment,
Indagationes Mathematicae 25 (2014) 785--799.

\bibitem{HMZ}
F.\ den Hollander, S.A.\ Molchanov and O.\ Zeitouni, 
{\emph{Random media at Saint-Flour. Reprints of lectures from the 
Annual Saint-Flour Probability Summer School held in Saint-Flour. 
Probability at Saint-Flour}}, Springer, Heidelberg, 2012.

\bibitem{HS}
F.\ Huveneers and F.\ Simenhaus,
Random walk driven by simple exclusion process,
preprint, arXiv:1404.4187 [math.PR].

\bibitem{JoRa11}
M.~Joseph and F.~Rassoul-Agha,
Almost sure invariance principle for continuous-space random walk in dynamic random environment,
ALEA Lat.\ Am.\ J.\ Probab.\ Math.\ Stat.\ 8 (2011) 43--57. 


\bibitem{LaLi}
G.F.\ Lawler and V.\ Limic,
\emph{Random Walk: A Modern Introduction},
Cambridge Studies in Advanced Mathematics 123,
Cambridge University Press, Cambridge, 2010.

\bibitem{Li}
T.\ Liggett,
\emph{Interacting Particle Systems},
Springer-Verlag, Berlin, 2005.

\bibitem{Lindvall79}
T.\ Lindvall, 
{On coupling of discrete renewal processes}, 
Z.\ Wahrscheinlichkeitstheorie verw.\ Gebiete 48 (1979) 57--70.

\bibitem{MoPe}
P.\ M{\"o}rters and Y.\ Peres,
\emph{Brownian Motion},
Cambridge Series in Statistical and Probabilistic Mathematics,
Cambridge University Press, Cambridge, 2010.

\bibitem{MoVa}
T.\ Mountford and M.E.\ Vares,
Random walks generated by equilibrium contact processes,
Electron.\ J.\ Probab.\ 20 no.\ 3 (2015) 1--17.

\bibitem{PeSiSoSt}
Y.\ Peres, A.\ Sinclair, P.\  Sousi and A.\ Stauffer,
Mobile geometric graphs: detection, coverage and percolation,
in: \textit{Proceedings of the Twenty-Second Annual ACM-SIAM
Symposium on Discrete Algorithms}, 412--428, SIAM,
Philadelphia, PA, 2011.

\bibitem{PoTe}
S.\ Popov and A.\ Teixeira,
Soft local times and decoupling of random interlacements.
Preprint 2012. [arXiv:1212.1605]

\bibitem{rghspp}
F.\ Rassoul-Agha, T.\ Sepp\"al\"ainen,  An almost sure invariance principle 
for random walks in a space-time random environment,
Probab.\ Theory Related Fields 133 (2005) 299--314.

\bibitem{ReVo13}
F.~Redig and F.~V\"ollering
Random walks in dynamic random environments: a transference principle,
Ann.\ Probab.\ 41 (2013) 3157--3180.

\bibitem{Re}
S.I.\ Resnick,
\emph{Extreme Values, Regular Variation and Point Processes},
Springer Series in Operations Research and Financial Engineering,
Springer, New York, 2008.

\bibitem{Ro}
Yu.A.\ Rozanov,
\emph{Markov Random Fields},
Applications of Mathematics,
Springer, New York, 1982.

\bibitem{Sar}
W.\ van Saarloos, Front propagation into unstable states, Phys.\ Rep.\ 386 (2003) 29--222.

\bibitem{Sz2}
A.S.\ Sznitman,
Topics in random walks in random environment,
School and Conference on Probability Theory, May 2002, ICTP Lecture Series,
Trieste, 203--266, 2004.

\bibitem{Ze98}
M.\ Zerner,
Lyapounov exponents and quenched large deviations for multidimensional random 
walk in random environment,
Ann.\ Probab.\ 26 (1998) 1407--1912.

\end{thebibliography}
\end{document}